\documentclass{article}
\usepackage[dvipdfm]{graphicx}
\usepackage[abs]{overpic}
\usepackage{graphicx,amsfonts,amsmath,amssymb}
\usepackage{slashbox}
\usepackage{multirow}
\usepackage{tabularx}
\usepackage{euscript}
\usepackage{eufrak}
\usepackage{fancybox}
\usepackage{float}
\usepackage{color}

\begin{document}

\newtheorem{theorem}{{\bf Theorem}}[section]
\newtheorem{fact}[theorem]{{\bf Fact}}
\newtheorem{lemma}[theorem]{{\bf Lemma}}
\newtheorem{sublemma}[theorem]{{\bf Sublemma}}
\newtheorem{proposition}[theorem]{{\bf Proposition}}
\newtheorem{corollary}[theorem]{{\bf Corollary}}
\newtheorem{definition}[theorem]{{\bf Definition}}
\newtheorem{notation}[theorem]{{\bf Notation}}
\newtheorem{convention}{{\bf Convention}}
\newtheorem{case}{{\bf Case}}
\newtheorem{casecase}{{\bf Case}}
\newtheorem{cacase}{{\bf Case}}
\newtheorem{subcacase}{{\bf Case 1.\hspace{-0.2em}}}
\newtheorem{subcase}{{\bf Case 1.\hspace{-0.2em}}}
\newtheorem{subcase2}{{\bf Case 2.\hspace{-0.2em}}}
\newtheorem{problem}{{\bf Problem}}
\newenvironment{proof}{{\bf Proof }}{\hfill$\square$}
\newcommand{\qed}{\hfill$\square$}
\newtheorem{remark}[theorem]{{\bf Remark}}%[section]
\newtheorem{claim}{{\bf Claim}}%[section]
\newcommand\Z{{\mathbb Z}}
\newcommand\N{{\mathbb N}}
\newcommand\Q{{\mathbb Q}}
\newcommand\R{{\mathbb R}}
\newcommand\diff{{\rm Diff}}
\newcommand\sym{{\rm Sym}}
\newcommand\fix{{\rm Fix}}
\newcommand\out{{\rm Out}}
\newcommand\aut{{\rm Aut}}
\newcommand\inn{{\rm Inn}}
\newcommand\id{{\rm id}}
\newcommand\K{{\mathcal K}}
\newcommand\T{{\mathcal T}}
\newcommand\LL{{\mathcal L}}
\newcommand\D{{\mathcal D}}
\newcommand\mcg{{\mathcal M}}
\newcommand\pmcg{{\mathcal PM}}
\newcommand\isom{{\rm Isom}}
\newcommand\interior{{\rm Int}}
\newcommand\cl{{\rm cl}}
\relpenalty=10000
\binoppenalty=10000
\uchyph=-1

\makeatletter
\def\tbcaption{\def\@captype{table}\caption}
\def\figcaption{\def\@captype{figure}\caption}
\makeatother

\title{Classification of 3-bridge spheres of 3-bridge arborescent links} 
\author{Yeonhee Jang
\footnote{This research is partially supported by Grant-in-Aid for JSPS Research Fellowships for Young Scientists.}\\
Department of Mathematics, Hiroshima University, \\
Hiroshima 739-8526, Japan\\
yeonheejang@hiroshima-u.ac.jp}

\date{\empty}
\maketitle

\begin{abstract}
In this paper, we give an isotopy classification of 3-bridge spheres of 3-bridge arborescent links, which are not Montesinos links. 
To this end, we prove a certain refinement of 
a theorem of J.S. Birman and H.M. Hilden \cite{Bir}
on the relation between bridge presentations of links 
and Heegaard splittings of 3-manifolds.
In the proof of this result, we also give 
an answer to a question by K. Morimoto \cite{Mor} 
on the classification of genus-2 Heegaard splittings of certain graph manifolds.
\end{abstract}

\section{Introduction}\label{intro}

An {\it $n$-bridge sphere} of a link $L$ in $S^3$
is a 2-sphere which meets $L$ in $2n$ points
and cuts $(S^3, L)$ into $n$-string trivial tangles
$(B_1, t_1)$ and $(B_2, t_2)$. 
Here, an {\it $n$-string trivial tangle} is 
a pair $(B^3, t)$ of the $3$-ball $B^3$ and 
$n$ arcs properly embedded in $B^3$ parallel to the boundary of $B^3$. 
We call a link $L$ an {\it $n$-bridge link} 
if $L$ admits an $n$-bridge sphere and does not admit an ($n-1$)-bridge sphere.
Two $n$-bridge spheres $S_1$ and $S_2$ of $L$ are said to be {\it pairwise isotopic} 
({\it isotopic}, in brief)
if there exists a homeomorphism $f:(S^3, L)\rightarrow (S^3, L)$
such that $f(S_1)=S_2$ and 
$f$ is {\it pairwise isotopic} to the identity,
i.e., there is a continuous family of homeomorphisms 
$f_t:(S^3, L)\rightarrow (S^3, L)$ $(0\leq t\leq 1)$
such that $f_0=f$ and $f_1=\id$.
Bridge numbers and bridge spheres of links have been studied in various references
(for example, see \cite{Boi3,Jan1,Jan2,Ota1,Ota2,Sch2,Sch} and references therein).

Recall that an {\it arborescent link} is a link obtained by closing an arborescent tangle
with a trivial tangle (see \cite{Con}),
where an {\it arborescent tangle} is a 2-string tangle obtained from rational tangles 
by repeatedly applying the operations in Figure \ref{fig-tanglesum}. 
%
%%%%%%%%%%%%%
\begin{figure}
\begin{center}
\includegraphics*[width=6.3cm]{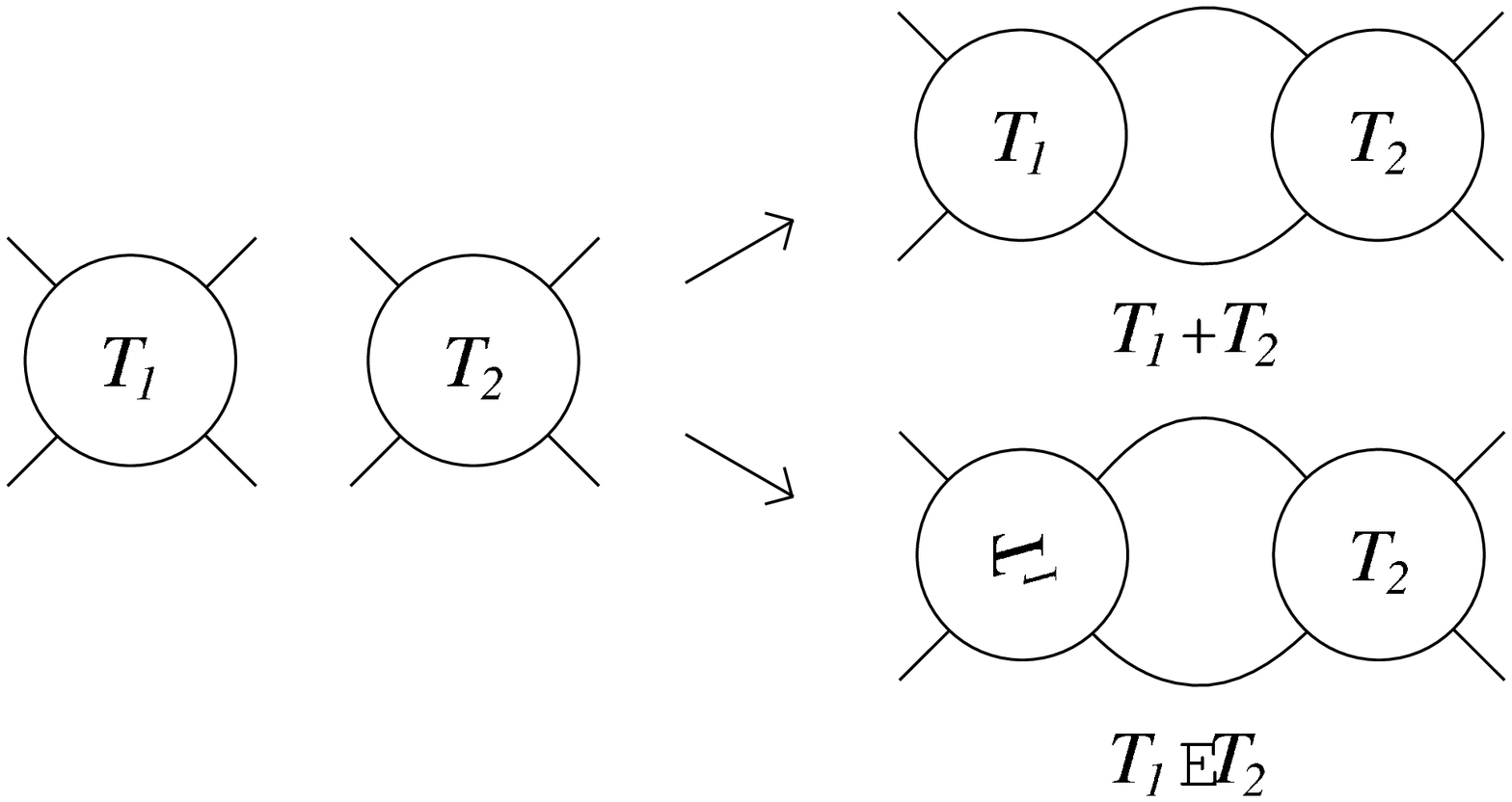}
\end{center}
\caption{}
\label{fig-tanglesum}
\end{figure}
%%%%%%%%%%%%%
%
These links form an important family of links
which contains 2-bridge links and Montesinos links,
and the double branched covering of the 3-sphere $S^3$ branched over an arborescent link
is a graph manifold.
Bonahon and Siebenmann \cite{Bon} 
gave a complete classification of arborescent links (cf. \cite{Fut}).

In \cite[Theorem 1]{Jan2}, we gave the following complete list of 3-bridge arborescent links,
where two links are {\it equivalent} 
if there exists an orientation-preserving homeomorphism of $S^3$
which carries one of the two links to the other.
(The classification of the links in the list up to equivalence is also given in \cite[Theorem 2]{Jan2}.)
\\

\begin{theorem}\label{thm-main-1}
Let $L$ be a 3-bridge arborescent link.
Then one of the followings holds.

%
%%%%%%%%%%%%%
\begin{figure}
\begin{center}
\includegraphics*[width=12cm]{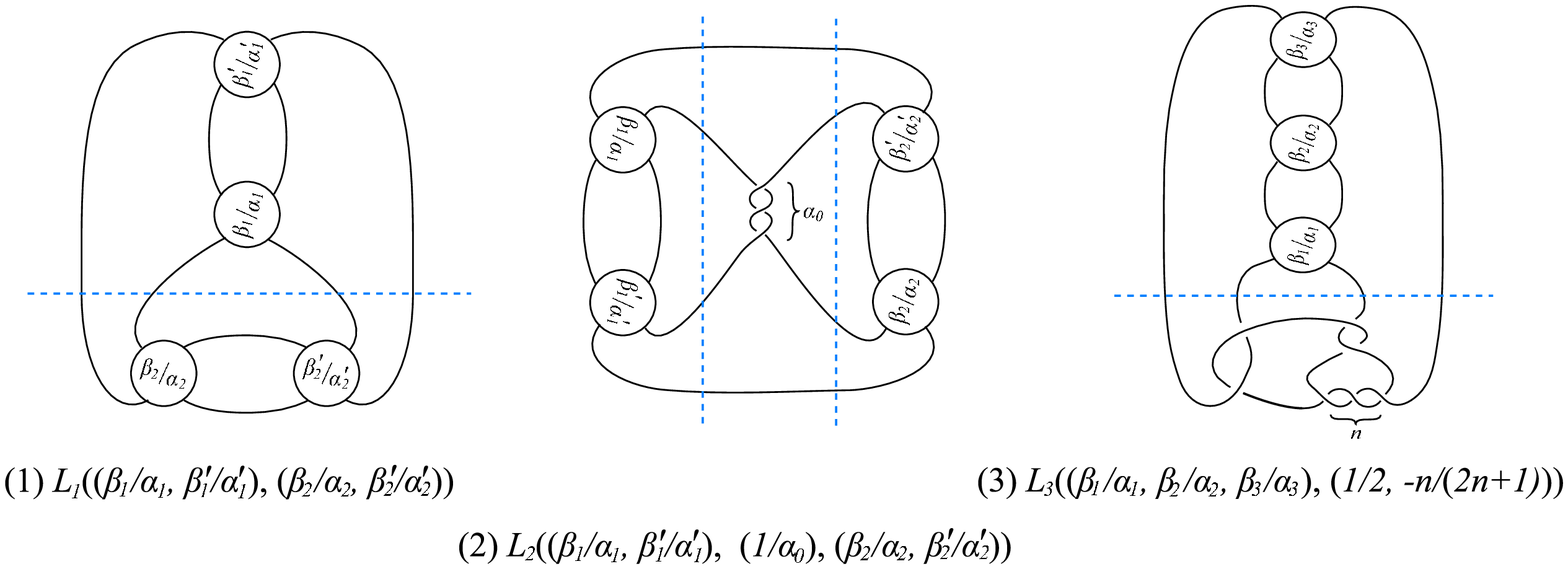}
\end{center}
\caption{}
\label{fig-3blinks}
\end{figure}
%%%%%%%%%%%%%
%
(1) $L$ is equivalent to the link 
$L_1((\beta_1/\alpha_1, \beta_1'/\alpha_1'),(\beta_2/\alpha_2, \beta_2'/\alpha_2'))$ 
in Figure \ref{fig-3blinks} (1).

(2) $L$ is equivalent to the link 
$L_2((\beta_1/\alpha_1, \beta_1'/\alpha_1'),(1/\alpha_0),(\beta_2/\alpha_2, \beta_2'/\alpha_2'))$ 
in Figure \ref{fig-3blinks} (2).

(3) $L$ is equivalent to the link 
$L_3((\beta_1/\alpha_1, \beta_2/\alpha_2, \beta_3/\alpha_3),(1/2, -n/(2n+1)))$ 
%with $|2n+1|\neq 1$
in Figure \ref{fig-3blinks} (3).

(4) $L$ is a Montesinos link 
$L(-b; \beta_1/\alpha_1, \beta_2/\alpha_2, \beta_3/\alpha_3)$.

Here, $\alpha_i$, $\alpha_i'$, $\beta_i$, $\beta_i'$ are integers 
such that $\alpha_i, \alpha_i'>1$ and 
$g.c.d.(\alpha_i, \beta_i)=g.c.d.(\alpha_i', \beta_i')=1$ ($i=1,2,3$), and
$\alpha_0$ and $n$ are integers such that
$|\alpha_0|>1$ and $|2n+1|>1$.
In Figure \ref{fig-3blinks},
the circle encircling a rational number $\beta/\alpha$
represents the rational tangle of slope $\beta/\alpha$.

\end{theorem}

For each $i=1,2,3$, we denote by $\LL_i$ the family of links as in ($i$) in the above theorem.
The main purpose of this paper is to give a complete classification 
of the 3-bridge spheres of the links in $\LL_1\cup \LL_2\cup \LL_3$.
We first present a complete list of the 3-bridge spheres.

\begin{theorem}\label{thm-main-2}
Any 3-bridge sphere of a link $L$ in $\LL_1\cup \LL_2\cup \LL_3$ is isotopic to 
one of the 3-bridge spheres in Figure \ref{fig-3bspheres}.
To be precise, the following hold.
\begin{itemize}
\item[{\rm (i)}] 
If $L$ belongs to $\LL_1$, 
then any 3-bridge sphere of $L$ is isotopic to one of the 3-bridge spheres $S_1$, $S_2$, $S_3$ and $S_4$ 
in (1), (2), (3) and (3') in Figure \ref{fig-3bspheres}.
\item[{\rm (ii)}] 
If $L$ belongs to $\LL_2$ or $\LL_3$, 
then any 3-bridge sphere of $L$ is isotopic to the 3-bridge sphere 
in (4) and (5) in Figure \ref{fig-3bspheres}, respectively.
\end{itemize}
%
%%%%%%%%%%%%%
\begin{figure}
\begin{center}
\includegraphics*[width=11cm]{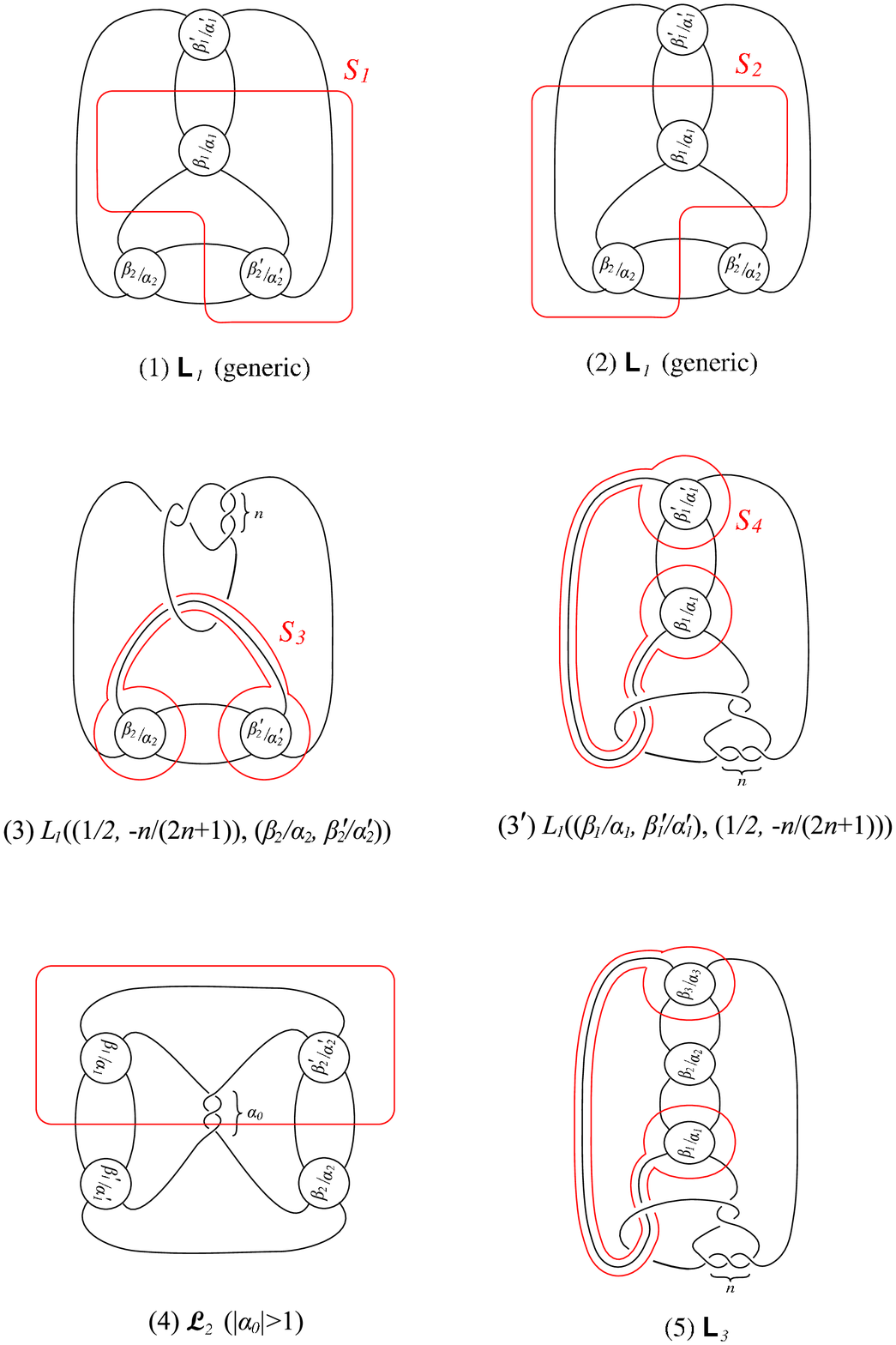}
\end{center}
\caption{}
\label{fig-3bspheres}
\end{figure}
%%%%%%%%%%%%%
%
\end{theorem}

\begin{remark}\label{rmk-i}
{\rm 
In the assertion (i),
the word \lq\lq generic\rq\rq\ in (1) and (2) in Figure \ref{fig-3bspheres}
means that every link $L\in\LL_1$ admits the 3-bridge spheres 
$S_1$ and $S_2$, respectively.
On the other hand, the 3-bridge spheres $S_3$ and $S_4$
in (3) and (3') are possessed only by the links
$L_1((1/2,-n/(2n+1)),(\beta_2/\alpha_2,\beta_2'/\alpha_2'))$ and
$L_1((\beta_1/\alpha_1,\beta_1'/\alpha_1'),(1/2,-n/(2n+1)))$,
respectively.
}
\end{remark}

Next we give 
a necessary and sufficient condition for any two 3-bridge spheres 
in Theorem \ref{thm-main-2}
to be isotopic.
This enables us to classify the 3-bridge spheres of
3-bridge arborescent links which are not Montesinos links.
In order to state the result,
we recall a notation introduced in \cite[Notation 1]{Jan2}.

\begin{notation}\label{notation}
{\rm
For rational numbers $s_1,\dots,s_r$ and $s_1',\dots,s_r'$,
%whose denominators are greater than $1$,
$$(s_1,\dots,s_r)\sim(s_1',\dots,s_r')$$ means that
$(s_1,\dots,s_r)=(s_1',\dots,s_r')$ or $(s_r',\dots,s_1')$ in $(\Q/\Z)^r$ and 
$\displaystyle\sum_{i=1}^rs_i=\sum_{i=1}^rs_i'$.
}
\end{notation}

\begin{theorem}\label{thm-distinction}
Two 3-bridge spheres $S_i$ and $S_j$ $(i,j\in\{1,2,3,4\},i\neq j)$ 
for a link
$L_1((\beta_1/\alpha_1,$ $\beta_1'/\alpha_1'), (\beta_2/\alpha_2, \beta_2'/\alpha_2'))$
are isotopic if and only if
$\{i,j\}=\{1,2\}$ and 
$
(\beta_k/\alpha_k, \beta_k'/\alpha_k')\sim$ $
(\varepsilon_k/\alpha_k, \varepsilon_k'/\alpha_k')
$
for some $k=1,2$,
where $\varepsilon_k, \varepsilon_k'\in\{\pm 1\}$.
\end{theorem}

From Theorems \ref{thm-main-2} and \ref{thm-distinction},
we obtain Table \ref{table-3bsphere},
which gives the number $\mu$ of isotopy classes of 3-bridge spheres of $L\in\LL_1$.
In Table \ref{table-3bsphere}, ($i$-$j$) $(i\in\{a,b\}, j\in\{1,2,3,4\})$ means that
$L$ satisfies the conditions $(i)$ and $(j)$ as follows.
\begin{itemize}
\item[(a)] $(\beta_k/\alpha_k, \beta_k'/\alpha_k')\sim
(\varepsilon_k/\alpha_k, \varepsilon_k'/\alpha_k')$
for some $k=1,2$,
where $\varepsilon_k, \varepsilon_k'\in\{\pm 1\}$,
\item[(b)] $(\beta_k/\alpha_k, \beta_k'/\alpha_k')\not\sim
(\varepsilon_k/\alpha_k, \varepsilon_k'/\alpha_k')$
for both $k=1,2$,
where $\varepsilon_k, \varepsilon_k'\in\{\pm 1\}$,
\item[(1)] $(\beta_1/\alpha_1, \beta_1'/\alpha_1')\sim(1/2, -n/(2n+1))$ for some $n$ and
$(\beta_2/\alpha_2, \beta_2'/\alpha_2')\not\sim(1/2, -m/(2m+$ $1))$ for any $m$,
\item[(2)] $(\beta_1/\alpha_1, \beta_1'/\alpha_1')\not\sim(1/2, -n/(2n+1))$ for any $n$ and
$(\beta_2/\alpha_2, \beta_2'/\alpha_2')\sim(1/2, -m/(2m+1))$ for some $m$,
\item[(3)] $(\beta_k/\alpha_k, \beta_k'/\alpha_k')\sim(1/2, -n/(2n+1))$ for some $n$ for each $k=1,2$,
\item[(4)] $(\beta_k/\alpha_k, \beta_k'/\alpha_k')\not\sim(1/2, -n/(2n+1))$ for any $n$ for each $k=1,2$.
\end{itemize}

\begin{table}
\begin{center}
\begin{tabular}{|l||c|c|c||c|}
\hline
{} & {$S_1, S_2$} & {$S_3$} & {$S_4$} & {$\mu$}
\\ \hline\hline
{(a-1)} & \multirow{4}{3pt}{1} & {1} & {0} & {$2$}
\\ 
{(a-2)} &  & {0} & {1} & {$2$}
\\ 
{(a-3)} &  & {1} & {1} & {$3$}
\\ 
{(a-4)} &  & {0} & {0} & {$1$}
\\ \hline
{(b-1)} & \multirow{4}{3pt}{2} & {1} & {0} & {$3$}
\\ 
{(b-2)} &  & {0} & {1} & {$3$}
\\ 
{(b-3)} &  & {1} & {1} & {$4$}
\\ 
{(b-4)} &  & {0} & {0} & {$2$}
\\ \hline
\end{tabular}
\caption{3-bridge spheres for $L\in \LL_1$}
\label{table-3bsphere}
\end{center}
\end{table}

In particular, we obtain the following corollary,
which gives an affirmative answer to a question raised by
Morimoto (see \cite[p.324]{Mor}).

\begin{corollary}\label{cor-4-3bspheres}
The link $L_1((1/2, -n/2n+1),(1/2, -m/2m+1))$ 
with $|2n+1|,|2m+1|\not\in \{1,3\}$ 
admits exactly four 3-bridge spheres up to isotopy 
(see Figure \ref{fig-4}).
%
%%%%%%%%%%%%%
\begin{figure}
\begin{center}
\includegraphics*[width=12cm]{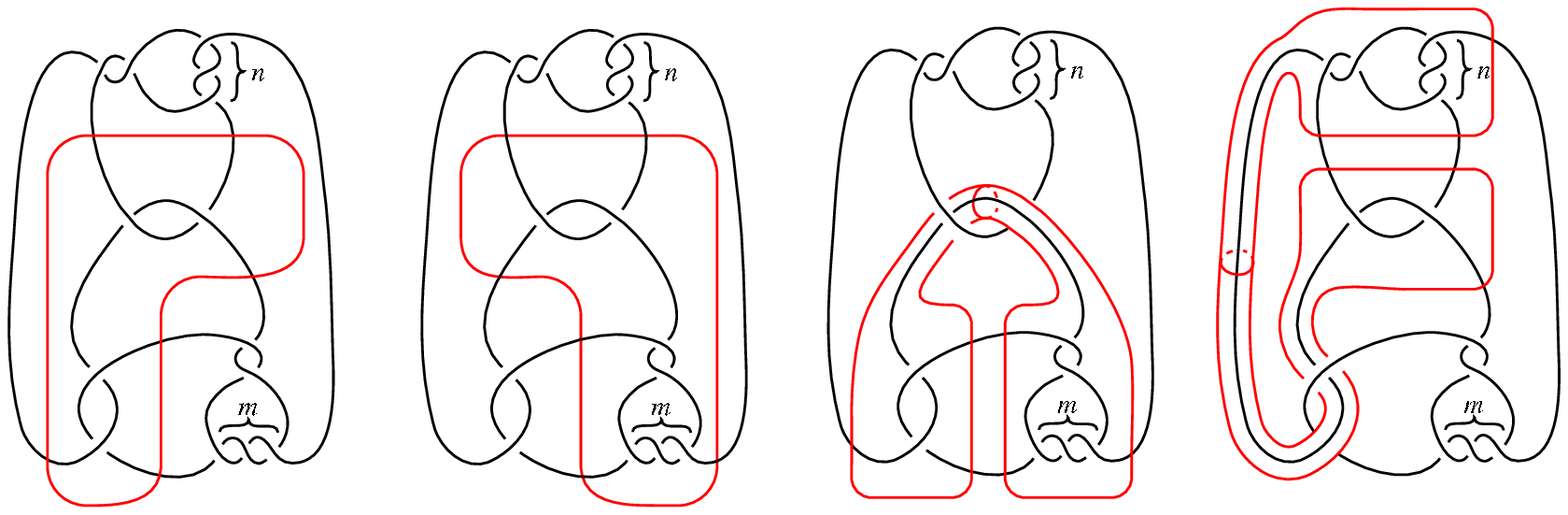}
\end{center}
\caption{}
\label{fig-4}
\end{figure}
%%%%%%%%%%%%%
%
\end{corollary}

Unfortunately, our methods do not work for Montesinos links.
However, we obtain the following partial result.

\begin{theorem}\label{thm-mont}
(1) A 3-bridge nonelliptic Montesinos link admits at most six 3-bridge spheres,
$P_1,\dots, P_6$, up to isotopy 
%
%%%%%%%%%
\begin{figure}
\begin{center}
\includegraphics*[width=12cm]{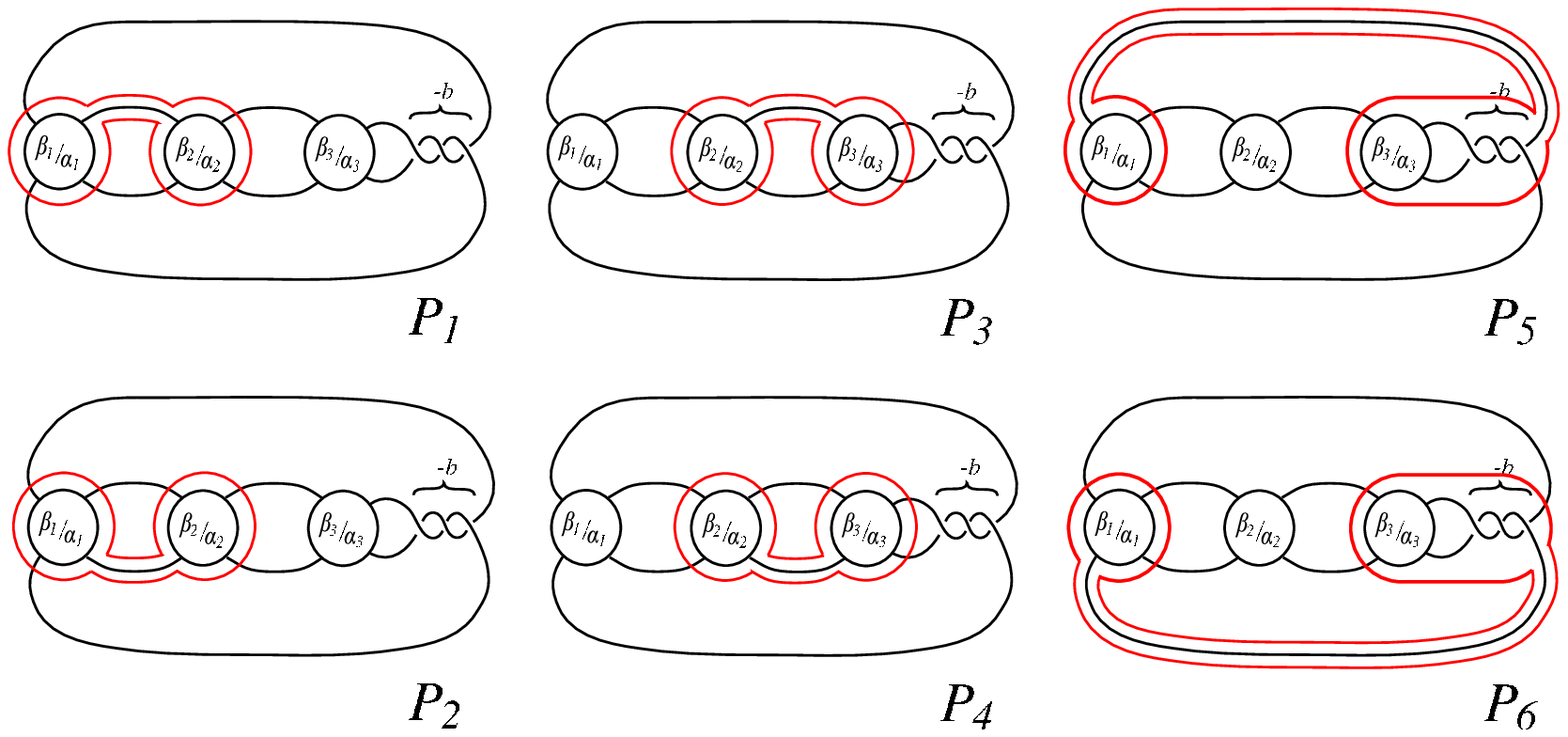}
\end{center}
\caption{}
\label{fig-mont-links}
\end{figure}
%%%%%%%%%
%
(see Figure \ref{fig-mont-links}). 

(2) A 3-bridge elliptic Montesinos link admits 
a unique 3-bridge sphere up to isotopy.
\end{theorem}

In the remainder of the introduction, we explain our strategy.
For a given 3-bridge link $L$,
we have a map $\Phi_L$
from the set of isotopy classes of 3-bridge spheres of $L$ 
to the set of isotopy classes of genus-2 Heegaard surfaces $F$ of the double branched covering $M_2(L)$, 
whose hyper-elliptic involutions $\tau_F$ are the covering transformation $\tau_L$.
\begin{eqnarray*}
\begin{array}{rll}
\Phi_L &:& \{{\rm 3\text{-}bridge\ spheres\ of\ }L\}/{\sim}\\
&&\rightarrow \{{\rm genus\text{-}2\ Heegaard\ surfaces\ }F\ {\rm of\ }M\ 
{\rm s.t.\ }\tau_F=\tau_L\}/{\sim}.
\end{array}
\end{eqnarray*}
It is obvious that $\Phi_L$ is surjective. 
In Section \ref{3b-hs}, we prove the following theorem
which gives a condition for $\Phi_L$ to be injective,
by using the results of Boileau and Zimmermann \cite{Boi6}.

\begin{theorem}\label{prop-hs-3b-0}
Let $L$ be a prime, unsplittable 3-bridge link.

(1) If $L$ is not a Montesinos link, then $\Phi_L$ is bijective.

(2) If $L$ is a nonelliptic Montesinos link,
then $\Phi_L$ is at most 2-1.
\end{theorem}

\begin{remark}
{\rm It follows from Theorem \ref{thm-mont}
that $\Phi_L$ is bijective when $L$ is an elliptic Montesinos link.
}
\end{remark}

By the above theorem, classification of 3-bridge spheres of 3-bridge arborescent links
(which are not Montesinos links)
is reduced to classification of genus-2 Heegaard surfaces of the double branched coverings.
A refinement of the results by Kobayashi \cite{Kob} and Morimoto \cite{Mor} (see \cite[Theorem 5]{Jan2})
enables us to obtain a complete list of genus-2 Heegaard surfaces.
To obtain a classification of the genus-2 Heegaard surfaces,
we use their commutator invariants (see Section \ref{classify}).
The commutator invariant turns out to be a complete invariant of 
genus-2 Heegaard splittings for genus-2 graph manifolds.
Namely, when given two 3-bridge spheres of a link cannot be distinguished by the commutator invariants, 
we can construct an isotopy between them.
This completes the classification of 3-bridge arborescent links, 
which are not Montesinos links,
and their 3-bridge spheres up to isotopy.

For Montesinos links, however, the pre-images of the 3-bridge spheres $P_i$ and $P_{i+1}$ $(i=1,3,5)$
are isotopic Heegaard surfaces, and thus
we cannot distinguish the 3-bridge spheres by the methods in this paper.
We also give, in Remark \ref{rmk-nonelliptic},
certain sufficient conditions for the 3-bridge spheres in Theorem \ref{thm-mont} (1) to be mutually isotopic.
We conjecture that these conditions actually provide
a necessary and sufficient condition.

This paper is organized as follows.
In Section \ref{3b-hs}, we prove Theorem \ref{prop-hs-3b-0}
on a relation between 3-bridge spheres and genus-2 Heegaard surfaces.
In Section \ref{sec-proof2}
we list all 3-bridge spheres of (non-Montesinos) arborescent links
up to isotopy
by using the results of Kobayashi \cite{Kob} and Morimoto \cite{Mor}.
This, together with the results in Section \ref{sec-simple-exception},
completes the proof of Theorem \ref{thm-main-2}.
In Section \ref{sec-pf-p0}, we prove Lemma \ref{lem-p0}
which is used in Section \ref{sec-proof2}.
In Section \ref{sec-simple-exception},
we prove that the \lq\lq simple exceptional links\rq\rq\  
admit a unique 3-bridge sphere up to isotopy.
In Section \ref{classify}, 
we use the commutator invariants of genus-2 Heegaard splittings
to distinguish two Heegaard surfaces up to isotopy.
In Section \ref{sec-mont}, 
we give the list of 3-bridge spheres of 3-bridge Montesinos links
and some sufficient conditions for them to be isotopic.

%%%%%(3-bridge spheres and genus-2 Heegaard surfaces)%%%%%%%%%%%%%%%%%%%%%%%%%%%%%%%%%%%%%%%%%%%%%%%%%%%%%%%%%%%%%%%%%%%%%%%

\section{3-bridge spheres and genus-2 Heegaard surfaces}\label{3b-hs}

Let $M$ be a closed orientable 3-manifold of Heegaard genus 2,
and let $(V_1, V_2; F)$ be a genus-2 {\it Heegaard splitting} of $M$,
i.e., $V_1$ and $V_2$ are genus-2 handlebodies in $M$
such that $M=V_1\cup V_2$ and $F=\partial V_1=\partial V_2=V_1\cap V_2$.
By \cite[Proof of Theorem 5]{Bir},
there is an involution $\tau$ on $M$
which satisfies the following condition.

\begin{itemize}
\item[($\ast$)] $\tau(V_i)=V_i$ $(i=1,2)$
and $\tau|_{V_i}$ is {\it equivalent} to the standard involution $\T$
on a standard genus-2 handlebody $V$ as illustrated in Figure \ref{fig-hs-3b}.
To be precise, there is a homeomorphism $\psi_i:V_i\rightarrow V$
such that $\T=\psi_i(\tau|_{V_i})\psi_i^{-1}$ $(i=1,2)$.
\end{itemize}
%
%%%%%%%%%%%%%
\begin{figure}
\begin{center}
\includegraphics*[width=9cm]{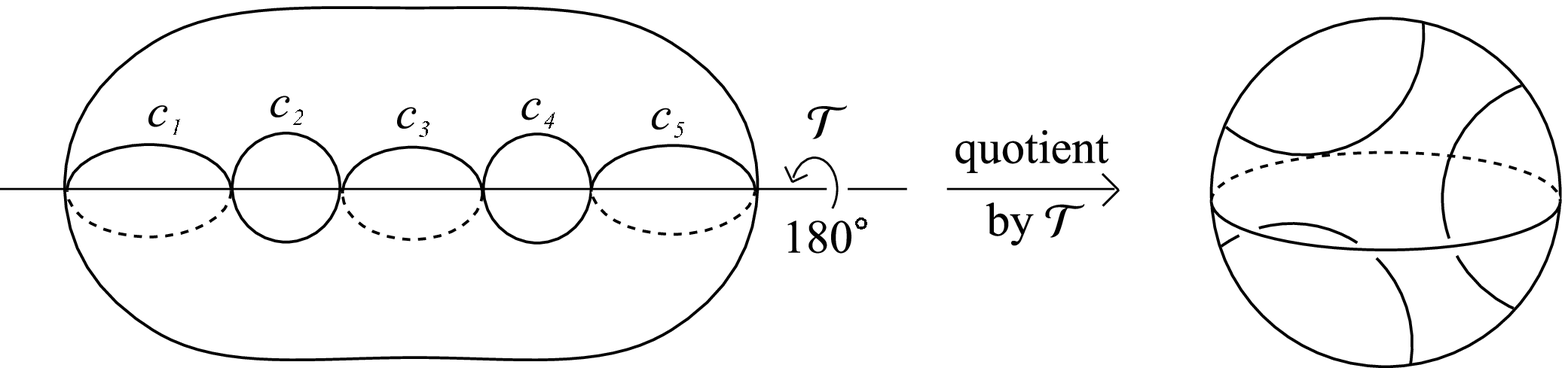}
\end{center}
\caption{}
\label{fig-hs-3b}
\end{figure}
%%%%%%%%%%%%%
%

For each genus-2 Heegaard splitting $(V_1, V_2; F)$,
we call an involution of $M$ satisfying the condition $(\ast)$
the {\it hyper-elliptic involution} 
associated with $(V_1, V_2; F)$ 
(or associated with $F$, in brief)
and denote it by $\tau_F$.
The strong equivalence class of $\tau_F$
is uniquely determined by the isotopy class of $(V_1, V_2; F)$
(see \cite[Proposition 5]{Jan2}).
Here, 
two involutions $\tau$ and $\tau'$ are said to be {\it strongly equivalent}
if there exists a homeomorphism $h$ on $M$ 
such that $h\tau h^{-1}=\tau'$ and that 
$h$ is isotopic to the identity map $\id_M$, and
two Heegaard splittings $(V_1, V_2; F)$ and $(W_1, W_2; G)$ of a 3-manifold $M$ 
are said to be {\it isotopic}
if there exists a self-homeomorphism $f$ of $M$ 
such that $f(F)=G$ and $f$ is isotopic to the identity map $\id_M$ on $M$.

Let $L$ be a 3-bridge link 
and let $M_2(L)$ be the double branched covering of $S^3$ 
branched over $L$.
Let $\tau_L$ be the covering transformation on $M_2(L)$.
If $S$ is a 3-bridge sphere of $L$, 
its pre-image in $M_2(L)$ is a genus-2 Heegaard surface $F$
such that $\tau_F=\tau_L$.
Moreover, the isotopy class of $F$ 
is uniquely determined by that of $S$
because an isotopy on $(S^3, L)$ lifts to an isotopy on $M_2(L)$.
Thus we obtain the following map $\Phi_L$
from the set of isotopy classes of 3-bridge spheres of $L$ 
to the set of isotopy classes of genus-2 Heegaard surfaces of $M_2(L)$, 
whose hyper-elliptic involutions are $\tau_L$.
\begin{eqnarray*}
\begin{array}{rll}
\Phi_L &:& \{{\rm 3\text{-}bridge\ spheres\ of\ }L\}/{\sim}\\
&&\rightarrow \{{\rm genus\text{-}2\ Heegaard\ surfaces\ }F\ {\rm of\ }M\ 
{\rm s.t.\ }\tau_F=\tau_L\}/{\sim}.
\end{array}
\end{eqnarray*}

It is obvious that $\Phi_L$ is surjective. 
In the following, we prove Theorem \ref{prop-hs-3b-0}
which gives a condition for $\Phi_L$ to be injective.

To prove Theorem \ref{prop-hs-3b-0}, 
we use a result by Boileau and Zimmermann \cite{Boi6}.
We recall a few concepts introduced in \cite{Boi6}.
Let $p:\widetilde{M}\rightarrow M$ be the universal covering of $M=M_2(L)$.
The group $O(L)$ generated by all lifts of $\tau_L$ to $\widetilde{M}$
is called the {\it $\pi$-orbifold group} of $L$.
The quotient space $\widetilde{M}/O(L)$ is
an orbifold with underlying space $S^3$ and singular set $L$.
A link in $S^3$ is said to be {\it sufficiently complicated} 
if it is prime, unsplittable and $O(L)$ is infinite.

\begin{remark}
{\rm 
The $\pi$-orbifold group $O(L)$ of a link $L$ 
is isomorphic to the quotient group 
$\pi_1(S^3\setminus L)/\langle\langle m^2 \rangle\rangle$, 
where $\langle\langle m^2 \rangle\rangle$ 
is the subgroup of $\pi_1(S^3\setminus L)$
normally generated by the square of the meridian of $L$
(cf. \cite[p.\,187]{Boi6})}.
\end{remark}

\begin{remark}
{\rm 
Let $L$ be an arborescent link.
Then, since $M=M_2(L)$ admits a reduced graph structure, 
%and admits a non-trivial torus decomposition,
$M$ is irreducible.
This implies that $L$ is prime and unsplittable
(cf. \cite[Proposition 1.1]{Boi6})}.
\end{remark}

We obtain the following lemma 
from the orbifold theorem \cite{Boi7,Coo} and a result of Dunbar \cite{Dun}.

\begin{lemma}\label{lem-sc}
Let $L$ be a prime, unsplittable link.
If $O(L)$ is finite, then $L$ is an elliptic Montesinos link.
\end{lemma}

\begin{proof}
Let $L$ be a prime, unsplittable link 
and assume that $O(L)$ is finite.
By the equivariant sphere theorem  
and the branched covering theorem (see \cite{Mor3}),
the orbifold $\widetilde{M}/O(L)$ is irreducible.
%$M=M_2(L)$ is irreducible (cf. \cite[Theorem 2]{Sak2}).
By the orbifold theorem \cite{Boi7,Coo},
$M$ is geometric and $\tau_L$ is an isometry of $M$.
Since $\pi_1(M)$ is finite by the assumption, 
$M$ is spherical.
Hence, the orbifold $\widetilde{M}/O(L)$ with underlying space $S^3$ 
and singular set $L$ is a spherical orbifold.
By Dunbar's classification of spherical 3-orbifolds (see Table 7 in \cite{Dun}),
we see that $L$ is an elliptic Montesinos link.
\end{proof}
\vspace{2mm}

Let $\gamma$ be the homomorphism
from the symmetry group $\sym (S^3, L)=\pi_0\diff (S^3, L)$ of $L$
to the outer automorphism group $\out O(L)=\aut O(L)/\inn O(L)$ 
defined by lifting diffeomorphisms and isotopies 
from $(S^3, L)$ to $\widetilde{M}$:
every lifted diffeomorphism induces an automorphism of $O(L)$ by conjugation.

\begin{proposition}\label{b-z}
(\cite[Theorem 2]{Boi6})
Let $L$ be a sufficiently complicated link in $S^3$. 
Then $\gamma : \sym(S^3, L)\rightarrow \out O(L)$ is an isomorphism.
\end{proposition}

We need the following lemma to prove Theorem \ref{prop-hs-3b-0}.

\begin{lemma}\label{condition4-1holds}
Let $M=M_2(L)$, $\tau_L$, $\widetilde{M}$, $O(L)$ be as above,
and assume that $L$ is sufficiently complicated.
Suppose that $\varphi$ is a self-homeomorphism of $M$ 
which is homotopic to the identity on $M$ and commutes with $\tau_L$.

(1) If $M$ is not a Seifert fibered space
such that the center of $\pi_1(M)$ is an infinite cyclic group,
then there exists a lift $\widetilde{\varphi}$ of $\varphi$ to $\widetilde{M}$
which induces by conjugation the identity on $O(L)$.

(2) If $M$ is a Seifert fibered space and 
the center of $\pi_1(M)$ is an infinite cyclic group, 
then there exists a lift $\widetilde{\varphi}$ of $\varphi$ to $\widetilde{M}$
which induces by conjugation the identity or $\alpha$ on $O(L)$.
Here, $\alpha$ is the automorphism on $O(L)$ given by
\begin{eqnarray*}
\alpha(x)=\left\{
\begin{array}{ll}
x & ([x,h]=1),\\
xh & ([x,h]\neq 1),
\end{array}
\right.
\end{eqnarray*}
where $h$ is an element of $O(L)$ representing a lift of a regular fiber of $M$.
\end{lemma}

\begin{proof}
This is proved in \cite[Proposition 4.12]{Boi6}
under the additional assumption that $M$ is Haken.
This assumption is used only to assure that
$M$ is a Seifert fibered space 
if the center of its fundamental group $\pi_1(M)$ is nontrivial.
However, by the affirmative answer
to the Seifert fibered space conjecture \cite[Corollary 8.3]{Gab},
the nontriviality of the center of $\pi_1(M)$ implies
that $M$ is a Seifert fibered space 
even if it is not Haken.
\end{proof}
\vspace{2mm}

{\sc Proof of Theorem \ref{prop-hs-3b-0}.}
%\begin{proof}{\bf of Theorem \ref{prop-hs-3b-0}}
Let $S$ and $S'$ be two 3-bridge spheres for 
a prime, unsplittable 3-bridge link $L$
and set $F:=p^{-1}(S)$ and $F':=p^{-1}(S')$, 
where $p$ is the covering projection.
Assume that the Heegaard surfaces $F$ and $F'$ of $M=M_2(L)$ are isotopic, 
namely, there is a homeomorphism $\varphi$ on $M$ sending $F$ to $F'$
which is isotopic to the identity map.
By the proof of \cite[Theorem 8]{Bir} 
(cf. the proof of \cite[Proposition 5]{Jan2}),
we may assume that $\varphi$ is $\tau_L$-equivariant.
Hence, we have a self-homeomorphism $\psi$ of $(S^3, L)$ 
which sends $S$ to $S'$ and lifts to $\varphi$.

(1) Assume that $L$ is not a Montesinos link.
Then $L$ is sufficiently complicated by Lemma \ref{lem-sc}. 

If $M$ is not a Seifert fibered space
such that $\pi_1(M)$ has an infinite cyclic center,
then there exists a lift $\widetilde{\varphi}$ of $\varphi$ to $\widetilde{M}$
which induces by conjugation the identity map on $O(L)$
by Lemma \ref{condition4-1holds} (1).
Hence $\varphi$ induces the identity map on $\out(O(L))$.
By Proposition \ref{b-z}, %\cite[Theorem 4.3]{Boi6}, 
$\psi$ is isotopic to the identity.
Hence $\Phi_L$ is injective.

If $M$ is a Seifert fibered space 
such that $\pi_1(M)$ has an infinite cyclic center,
then $L$ is a Seifert link, namely,
$S^3\setminus L$ admits a Seifert fibration by circles
(see, for example, \cite[proof of Theorem 1.3]{Boi6}).
By \cite[Corollary 1.4]{Boi6},
$O(L)$ has a nontrivial center.
Hence, by \cite[Proposition 4.12]{Boi6},
there exists a lift $\widetilde{\varphi}$ of $\varphi$ to $\widetilde{M}$
which induces by conjugation the identity map on $O(L)$.
As in the previous case, we see by using Proposition \ref{b-z}
that $\Phi_L$ is injective.

(2) Assume that $L$ is a nonelliptic Montesinos link.
Then $L$ is sufficiently complicated by Lemma \ref{lem-sc}, 
and $M$ is a Seifert fibered space.
Note that the center of $\pi_1(M)$ is an infinite cyclic group
whose generator is a regular fiber.
By Lemma \ref{condition4-1holds} (2), 
there exists a lift $\widetilde{\varphi}$ of $\varphi$ to $\widetilde{M}$
which induces by conjugation the identity map or $\alpha$, 
as in Lemma \ref{condition4-1holds} (2), on $O(L)$.
Hence, by Proposition \ref{b-z}, %\cite[Theorem 4.3]{Boi6}, 
we have at most two 3-bridge spheres up to isotopy 
for each isotopy class of genus-2 Heegaard surfaces.
\qed
%\end{proof}

\begin{remark}\label{remark-tau}
{\rm 
We note that the automorphism $\alpha$ of $O(L)$
in the above proof of Theorem \ref{prop-hs-3b-0} (2)
is induced by a lift of the symmetry $\rho$ of $(S^3, L)$,
in Figure \ref{fig-mont-inv}, 
to the universal cover $\widetilde{M}$ of $M=M_2(L)$.
%
%%%%%%%%%%%%%
\begin{figure}
\begin{center}
\includegraphics*[width=7cm]{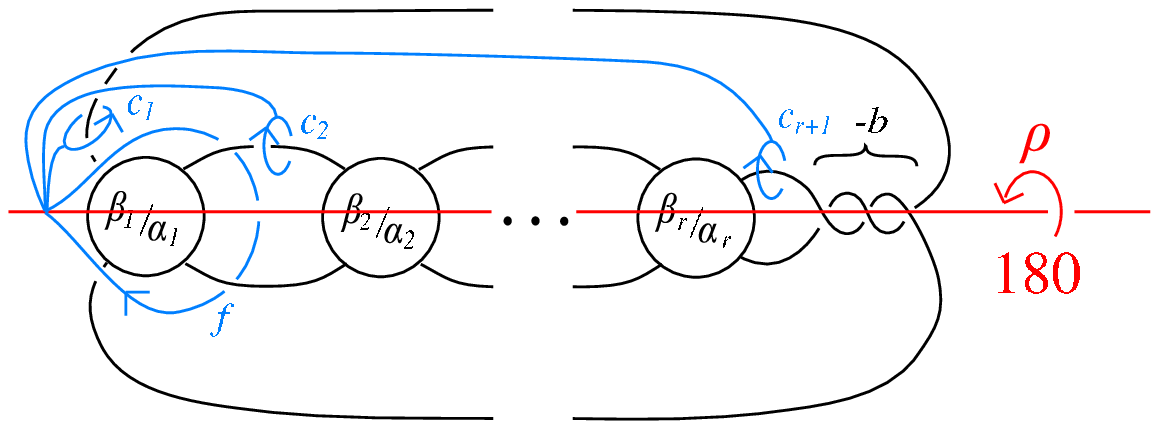}
\end{center}
\caption{}
\label{fig-mont-inv}
\end{figure}
%%%%%%%%%%%%%
%
%
In fact, 
if $L$ is a Montesinos link 
$L(-b; \beta_1/\alpha_1, \beta_2/\alpha_2,\dots ,\beta_r/\alpha_r)$
(see \cite[Section 2]{Jan2} for notation), 
then $O(L)\cong \pi_1(S^3\setminus L)/\langle\langle m^2 \rangle\rangle$ 
has a group presentation
$$
O(L)=
\langle
c_1, c_2, \dots, c_{r+1}, f \mid
c_i^2, c_ifc_i^{-1}f, %(i=1,\dots,r+1), 
(c_jc_{j+1})^{\alpha_j}f^{\beta_j}, %(j=1,\dots,r), 
c_1c_{r+1}f^b
\rangle ,
$$
where $c_i$ and $f$ are represented by the loops $c_i$ and $f$, respectively,
in Figure \ref{fig-mont-inv}
(cf. \cite{Boi5}).
Let $\widetilde{\rho}$ be a lift of $\rho$ to $\widetilde{M}$, 
and let $\iota_{\widetilde{\rho}}$ be the automorphism of $O(L)$
induced by conjugation by $\widetilde{\rho}$.
Then we can observe that 
$\iota_{\widetilde{\rho}}(c_1)=c_1^{-1}f=c_1f$,
$\iota_{\widetilde{\rho}}(c_j)=(c_1f)(c_jf)(c_1f)^{-1}$ $(j=2,\dots, r+1)$
and $\iota_{\widetilde{\rho}}(f)=(c_1f)f(c_1f)^{-1}$.
Thus $\iota_{\widetilde{\rho}}$ and $\alpha$
determine the same element of $\out O(L)$.
}
\end{remark}

%%%%%(Proof of Theorem \ref{thm-main-2})%%%%%%%%%%%%%%%%%%%%%%%%%%%%%%%%%%%%%%%%%%%%%%%%%%%%%%%%%%%%%%%%%%%%%%%

\section{Proof of Theorem \ref{thm-main-2}}\label{sec-proof2}

We quickly recall several notations from \cite{Jan2}.
The symbol $F(\beta_1/\alpha_1,\dots,\beta_r/\alpha_r)$ with $F=D$, $A$ or $M\ddot{o}$
denotes a Seifert fibered space over a disk, an annulus or a M\"{o}bius band
with singular fibers of indices $\beta_1/\alpha_1,\dots,\beta_r/\alpha_r$.
Each boundary component of $F(\beta_1/\alpha_1,\dots,\beta_r/\alpha_r)$
has a {\it horizontal loop}
which intersects each regular fiber on the boundary component transversely in a single point.
See \cite[Section 2]{Jan2} for precise definition.
For a knot or a link $K$ in a manifold, $E(K)$ denotes the exterior of $K$.

%\vspace{2mm}

%\begin{proof}{\bf of Theorem \ref{thm-main-2}} 
%
Let $L$ be a 3-bridge arborescent link
which belongs to one of the families $\LL_1$, $\LL_2$ and $\LL_3$ 
in Theorem \ref{thm-main-1},
and let $S$ be a 3-bridge sphere of $L$.
Then $L$ is prime, unsplittable and is not a Montesinos link (cf. \cite[Theorem 2]{Jan2}).
Thus $L$ satisfies the assumption of Theorem \ref{prop-hs-3b-0} (1).
Hence the isotopy class of $S$ is uniquely determined by
the isotopy class of the genus-2 Heegaard surface $F$, obtained as the pre-image of $S$, 
of the double branched covering $M=M_2(L)$
such that the hyper-elliptic involution $\tau_F$ associated with $F$
is identified with the covering involution $\tau_L$ on $M$.

\begin{casecase}\label{casecase1}
$L=L_1((\beta_1/\alpha_1, \beta_1'/\alpha_1'),(\beta_2/\alpha_2, \beta_2'/\alpha_2'))\in\LL_1$.
\end{casecase}

Then 
by the Montesinos trick \cite{Mon3} (cf. \cite[Proposition 7]{Jan2}),
we see that the double branched cover $M$ of $S^3$ branched along $L$
is obtained from the Seifert fibered spaces $M_1=D(\beta_1/\alpha_1, \beta_1'/\alpha_1')$
and $M_2=D(\beta_2/\alpha_2, \beta_2'/\alpha_2')$,
by gluing them along their boundaries by a homeomorphism 
so that a horizontal loop and a regular fiber of $M_1$ are identified with a regular fiber and a horizontal loop of $M_2$, respectively.
%(see \cite[Section 2]{Jan2} for notation of Seifert fibered spaces and the definition of horizontal loops).
Let $F_1$ and $F_2$ be the genus-2 Heegaard surfaces of $M$ 
obtained as the pre-images of the 3-bridge spheres $S_1$ and $S_2$ 
in Figure \ref{fig-3bspheres}, respectively.
When $M_1$ (resp. $M_2$) is homeomorphic to $D(1/2, -n/2n+1)$ 
for some integer $n$ with $|2n+1|>1$,
let $F_3$ (resp. $F_4$) be the genus-2 Heegaard surface of $M$ 
obtained as the pre-image of the 3-bridge spheres $S_3$ (resp. $S_4$).
Note that $F_1$ and $F_2$ are the two genus-2 Heegaard surfaces of $M$ belonging to the family F(1) in \cite{Mor} (cf. \cite[Proposition 5.2]{Mor} and \cite[Section 7, Case 1.1]{Jan2})
and that $F_3$ and $F_4$ are genus-2 Heegaard surfaces of $M$ belonging to the families F(2-1) and F(2-2) in \cite{Mor}, respectively.
By \cite[Section 5]{Mor}, 
we see that $F$ is isotopic to $F_1$, $F_2$, $F_3$ or $F_4$.
Hence, by Theorem \ref{prop-hs-3b-0},
$S$ is isotopic to $S_1$, $S_2$, $S_3$ or $S_4$.

\begin{casecase}\label{casecase2}
$L=L_2((\beta_1/\alpha_1, \beta_1'/\alpha_1'),(1/\alpha_0),(\beta_2/\alpha_2, \beta_2'/\alpha_2'))\in\LL_2$.
\end{casecase}

By \cite[Proposition 4]{Jan2} together with the fact that $\LL_1$, $\LL_2$ and $\LL_3$ are mutually disjoint (see \cite[Theorem 2]{Jan2}),
one of the following holds.
\begin{itemize}
\item[(i)] $L$ is equivalent to the link $L_2((-1/2,1/2), (1/n),(-1/2,1/2))\in \LL_2$
in Figure \ref{fig-gen-mont}.
%
%%%%%%%%%%%%%
\begin{figure}
\begin{center}
\includegraphics*[width=3cm]{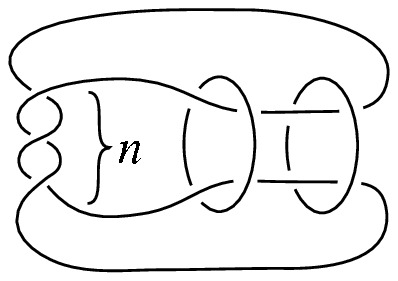}
\end{center}
\caption{$n\neq 0$}
\label{fig-gen-mont}
\end{figure}
%%%%%%%%%%%%%
%
In this case, $L$ is non-simple, i.e.,  $S^3\setminus L$ contains an essential torus.
\item[(ii)] The double branched covering $M=M_2(L)$ is a graph manifold which admits a nontrivial torus decomposition by separating tori.
\end{itemize}

The exceptional case where the condition (i) holds
is treated in Section \ref{sec-simple-exception},
where we show that $L$ admits a unique 3-bridge sphere up to isotopy
(Proposition \ref{lem-unique-sp}).
Thus we may assume that $M$ satisfies the condition (ii).

Then the 3-manifold $M$ and its genus-2 Heegaard surface $F$ satisfies the assumption of \cite[Proposition 7]{Jan2},
and we see from the proposition that 
one of the following conditions (a) and (b) holds.
\begin{itemize}
\item[(a)] $M$ belongs to the family M(1-b) and 
$F$ satisfies the condition (F1) in \cite[Theorem 5 and Definition 1]{Jan2}.
Namely, 
\begin{itemize}
\item[(M1-b)] 
$M$ is obtained by gluing $M_1=D(\beta_1/\alpha_1, \beta_1'/\alpha_1')$
and $M_2=E(K)=M\ddot{o}(1/\alpha_0)$, where $K$ is a 1-bridge knot in a lens space, so that
a horizontal loop and a regular fiber of $M_1$ are identified with a regular fiber and a horizontal loop of $M_2$, respectively,
and 
\item[(F1)] the intersection of the torus $T:=\partial M_1=\partial M_2$ 
and each handlebody bounded by $F$ is a single separating essential annulus (see Figure \ref{fig-hs} (F1)).
\end{itemize}
%
%%%%%%%%%%%%%
\begin{figure}
\begin{center}
\includegraphics*[width=13.5cm]{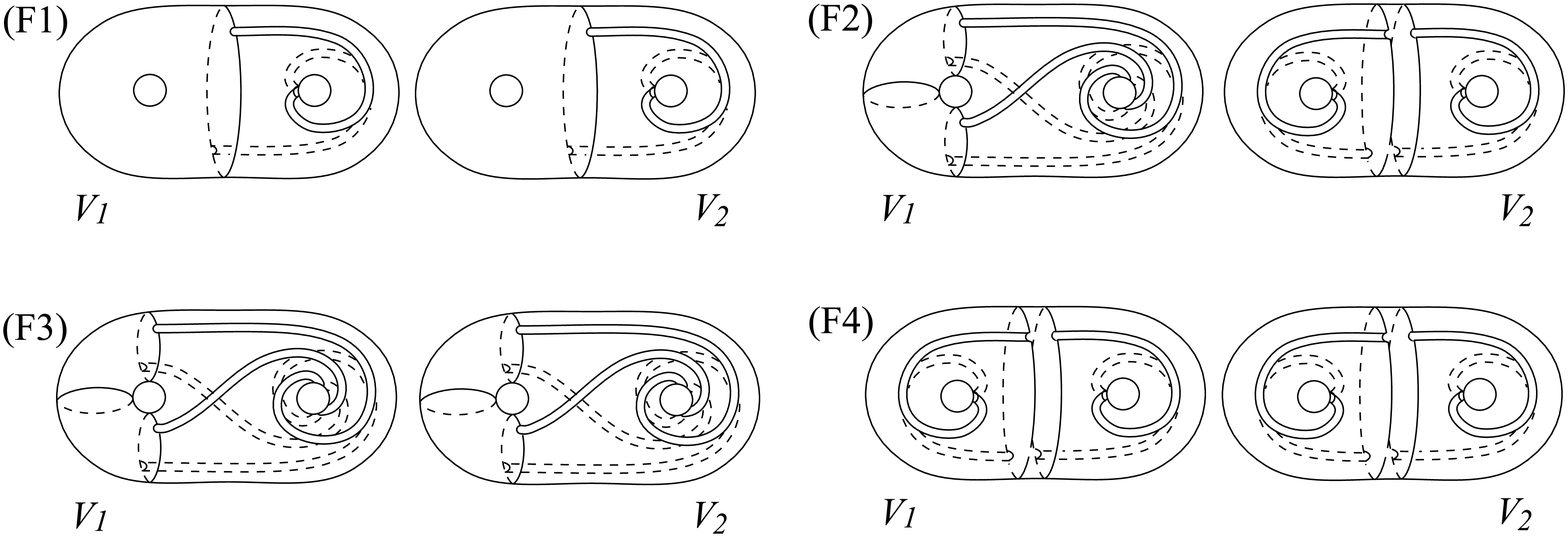}
\end{center}
\caption{}
\label{fig-hs}
\end{figure}
%%%%%%%%%%%%%
%
Moreover, 
\begin{itemize}
\item $M_1\cap F$ is an essential annulus saturated in the Seifert fibration of $M_1$, and
\item $M_2\cap F$ is a 2-holed torus which gives a 1-bridge decomposition of 
the 1-bridge knot $K$ such that $M_2=E(K)$.
\end{itemize}
\item[(b)] 
$M$ belongs to the family M(4) and 
$F$ satisfies the condition (F4) in \cite[Theorem 5]{Jan2}.
Namely, 
\begin{itemize}
\item[(M4)] 
$M$ is obtained by gluing 
$M_1=D(\beta_1/\alpha_1, \beta_1'/\alpha_1')$, $M_2=D(\beta_2/\alpha_2, \beta_2'/\alpha_2')$
and $M_3=E(S(2\alpha_0,1))=A(1/\alpha_0)$, 
where $S(2\alpha_0,1)$ is the 2-bridge link of type $(2\alpha_0,1)$,
so that 
a horizontal loop and a regular fiber of $M_i$ $(i=1,2)$ are identified with a regular fiber and a horizontal loop of $M_3$, respectively,
and
\item[(F4)] the intersection of the pair of tori $T:=\partial (M_1\cup M_2)=\partial M_3$ 
and each handlebody bounded by $F$ consists of 
two disjoint non-parallel separating essential annuli (see Figure \ref{fig-hs} (F4)).
\end{itemize}
Moreover, 
\begin{itemize}
\item $M_i\cap F$ is an essential saturated annulus in $M_i$ $(i=1,2)$, and
\item $M_3\cap F$ is a 2-bridge sphere.
\end{itemize}
\end{itemize}

Suppose that the condition (a) holds.
By \cite[Lemma 9]{Jan2},
we see that $\tau_L|_{M_2}$ is equivalent to the involution $g_1$ in \cite[Lemma 4 (2)]{Jan2},
where $(M_2, \fix g_1)/\langle g_1\rangle$ is the Montesinos pair
as illustrated in Figure \ref{fig-gamma} (1).
Recall from \cite[Remark 6]{Jan2} that the lens space containing $K$ is
homeomorphic to $P^2(0;1/\alpha_0)\cong S^2(\alpha_0;-1/2,1/2)$ and
that $K$ is a regular fiber of $P^2(0;1/\alpha_0)$ 
and the meridian of $K$ is a horizontal loop of $M_2=M\ddot{o}(1/\alpha_0)$.
%(see \cite[Section 2]{Jan2} for notation).
Then $g_1$ extends to an involution,
denoted by the same symbol $g_1$,
of the regular neighborhood $N(K)$ of $K$ in the lens space
such that $(N(K), \fix g_1, K)/\langle g_1\rangle$
is as illustrated in Figure \ref{fig-gamma} (2).
%
%%%%%%%%%%%%%
\begin{figure}
\begin{center}
\includegraphics*[width=6.2cm]{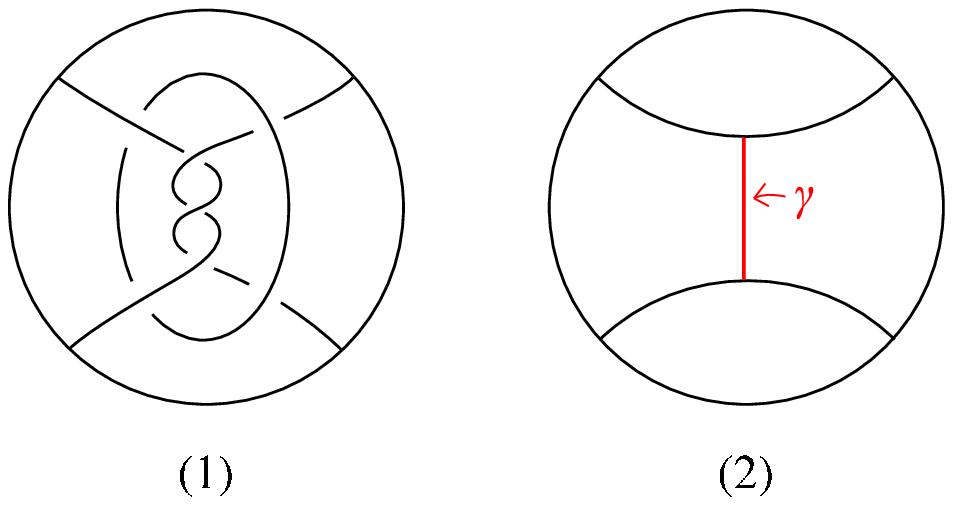}
\end{center}
\caption{}
\label{fig-gamma}
\end{figure}
%%%%%%%%%%%%%
%
Here, $K/\langle g_1\rangle$ is identified with the arc $\gamma$ in the figure.
Since the meridian of $K$ is identified with 
the horizontal loop of $M_2=M\ddot{o}(1/\alpha_0)$,
the lens space $P^2(0;1/\alpha_0)$ 
is the double branched covering of $S^3$ 
branched over the link $L'$ in Figure \ref{fig-p}
and the image of $K$ by the covering projection is the arc $\gamma$
in Figure \ref{fig-p}.
Recall that $F\cap M_2$ is a 2-holed torus 
which gives a 1-bridge decomposition of $K$
and that $\tau=\tau_F$ preserves $M_2$ and $F$ (cf. Figure \ref{fig-hs} (F1)).
Thus $F\cap M_2$ projects to a surface, say $\check{P}$, 
in $M_2/\langle g_1\rangle$
such that $\partial\check{P}$ is a simple loop 
on $\partial(N(K)/\langle g_1\rangle)$ of \lq\lq slope\rq\rq\  0.
Thus $\partial\check{P}$ bounds a disk in $N(K)/\langle g_1\rangle$
intersecting $\gamma$ transversely in a single point.
Let $P$ be the union of $\check{P}$ and the disk.
Then we have the following lemma, which we prove in Section \ref{sec-pf-p0}.
\begin{lemma}\label{lem-p0}
Under the above setting, 
$P$ is isotopic to the surface $P_0$ in Figure \ref{fig-p} 
by an isotopy of $(S^3, L')$ preserving $\gamma$.
%
%%%%%%%%%%%%
\begin{figure}
\begin{center}
\includegraphics*[width=4.2cm]{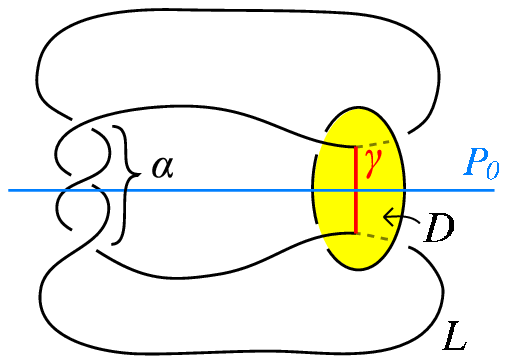}
\end{center}
\caption{}
\label{fig-p}
\end{figure}
%%%%%%%%%%%%
%
\end{lemma}
Thus $\check{P}$ is isotopic to the disk in Figure \ref{fig-piece3-1} (1).
%
%%%%%%%%%%%%%
\begin{figure}
\begin{center}
\includegraphics*[width=8cm]{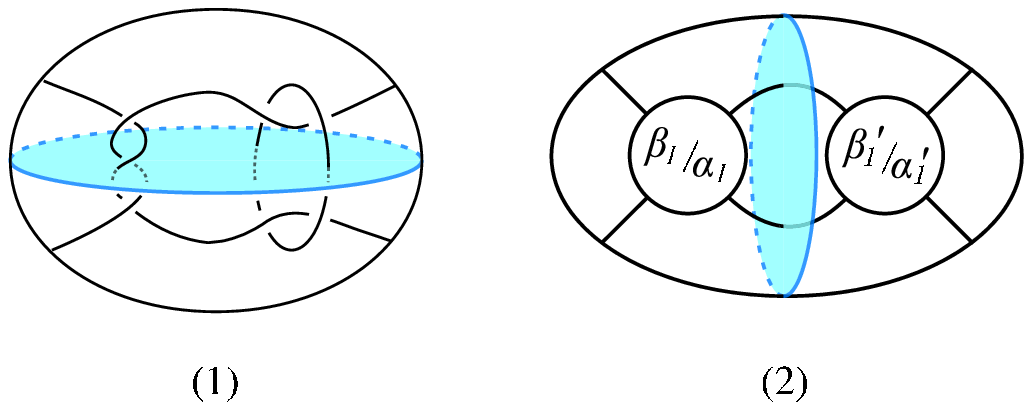}
\end{center}
\caption{}
\label{fig-piece3-1}
\end{figure}
%%%%%%%%%%%%%
%
On the other hand, we can see that $F\cap M_1$ projects to the disk in $(M_1, \fix \tau|_{M_1})/\langle \tau|_{M_1}\rangle$ as illustrated in Figure \ref{fig-piece3-1} (2).
Hence, by \cite[Lemma 7]{Jan2}, 
$S$ is isotopic to the 3-bridge sphere in Figure \ref{fig-3b_sphere2} 
and hence we obtain the desired result.
%
%%%%%%%%%%%%%
\begin{figure}
\begin{center}
\includegraphics*[width=8cm]{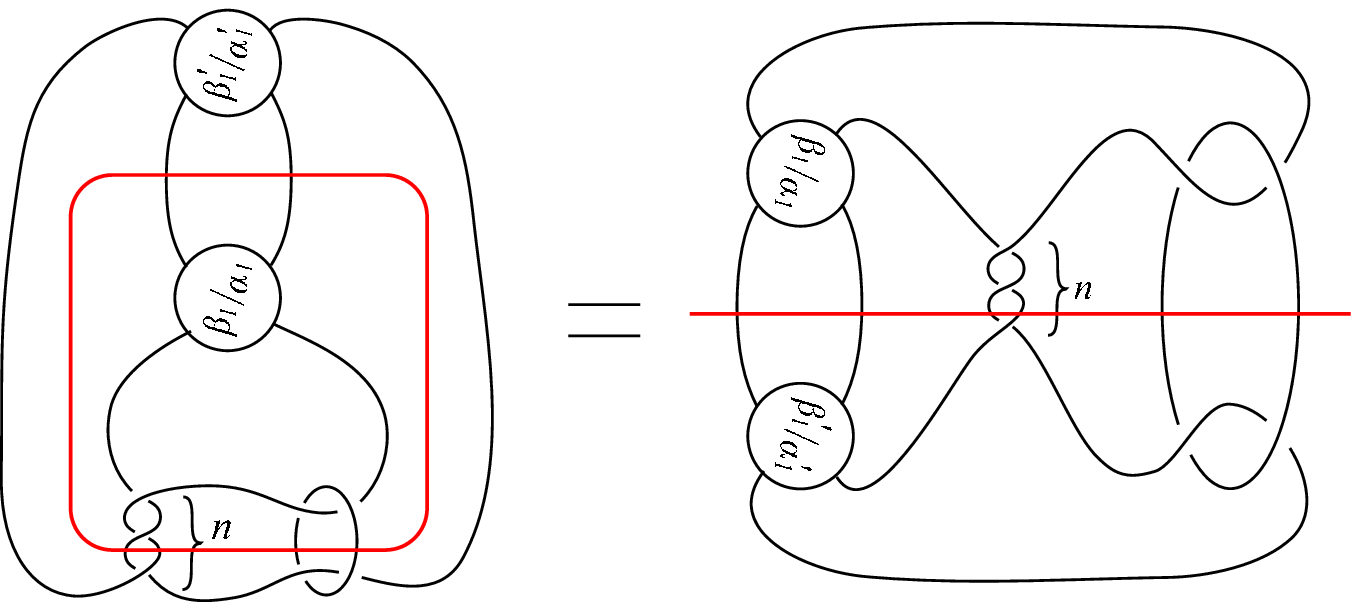}
\end{center}
\caption{}
\label{fig-3b_sphere2}
\end{figure}
%%%%%%%%%%%%%
%
%(cf. \cite[Figure 21]{Jan2}).

Suppose that the condition (b) holds.
Let $F_0$ be the pre-image of the 3-bridge sphere in Figure \ref{fig-3bspheres} (4).
Then, by the argument in Case 4 in \cite[Section 7]{Jan2},
the Heegaard surface $F$ obtained as the pre-image of the given bridge sphere $S$ is isotopic to $(D_{\mu}^{1/2})^n(F_0)$,
the surface obtained from $F_0$ by applying $n/2$-Dehn twist along $T=\partial M_1=\partial M_2$
in the direction of the regular fiber $\mu$ of $M_2$ 
(see \cite[Section 6]{Jan2} for the precise definition).
We may assume that $\tau_{F_0}$ is equal to the homeomorphism $G_1$ in \cite[Proposition 6 (3)]{Jan2}.
Then $\tau_F=D_{\mu}^n\tau_{F_0}$ by \cite[Lemma 5]{Jan2}.
Since $(D_{\mu})^n\neq 1$ whenever $n\neq 0$ by \cite[Lemma 3 (1)]{Jan2},
we see the identity $\tau_F=\tau_L(=\tau_{F_0})$
holds only when $n=0$.
Hence, by Theorem \ref{prop-hs-3b-0},
$S$ is isotopic to the 3-bridge sphere in Figure \ref{fig-3bspheres} (4).

\begin{casecase}\label{casecase3}
$L=L_3((\beta_1/\alpha_1, \beta_2/\alpha_2, \beta_3/\alpha_3),(1/2,-n/(2n+1)))\in\LL_3$.
\end{casecase}

By \cite[Proposition 7]{Jan2},
$M$ belongs to the family (M2-b) and $F$ satisfies the condition (F2) in \cite[Theorem 5]{Jan2}.
Namely,
\begin{itemize}
\item[(M2-b)] $M$ is obtained from
$M_1=D(\beta_1/\alpha_1, \beta_2/\alpha_2, \beta_3/\alpha_3)$
and $M_2=E(S(2n+1,1))\cong$ $D(1/2,-n/(2n+1))$
by gluing their boundary so that a horizontal loop and a regular fiber of $M_1$ are identified with a regular fiber and a horizontal loop of $M_2$, respectively, and
\item[(F2)] the intersection of the torus $T:=\partial M_1=\partial M_2$ 
and each handlebody bounded by $F$
consists of two essential annuli as illustrated in Figure \ref{fig-hs} (F2).
Moreover, 
\begin{itemize}
\item $M_1\cap F$ consists of two disjoint essential saturated annuli in $M_1$ 
which divide $M_1$ into three solid tori, and
\item the 2-bridge knot corresponding to $M_2$ is $S(2n+1,1)$, and
$M_2\cap F$ is a 2-bridge sphere.
\end{itemize}
\end{itemize}

By \cite[Theorem 4]{Mor}, a 2-bridge sphere of a 2-bridge knot $S(2n+1,1)$
is unique up to isotopy fixing the knot.
So, by \cite[Lemma 6 (2)]{Jan2},
the isotopy type of $F$ is uniquely determined by the isotopy type of $M_1\cap F$,
where the isotopy does not necessarily fix the boundary of $M_1$.

In order to determine if $\tau_F=\tau_L$,
we quickly recall some notations of certain subgroups of the mapping class groups of $M$ and $M_1$ introduced in \cite{Jan2}.
Let $\mcg(M_1)$ be the subgroup of the (orientation-preserving) mapping class group of $M_1$ 
which consists of the elements preserving each singular fiber of $M_1$.
(See \cite[Section 5]{Jan2} for more details.)
Let $\mcg(M)$ be the subgroup of the (orientation-preserving) mapping class group of $M$ 
which consists of the elements preserving each $M_i$ and each singular fiber of $M_i$ ($i=1,2$).
%Let $\mcg_0(M)$ be the subgroup of $\mcg(M)$
%which consists of the elements whose restriction to $M_2$ is the identity.
Throughout this paper, we do not distinguish between a self-homeomorphism and its isotopy class:
we denote them by the same symbol.

Note that $F\cap M_1$ is homeomorphic to 
one of the saturated annuli $G_1$, $G_2$ and $G_3$ 
obtained as the pre-images of the arcs in the base orbifold illustrated in Figure \ref{fig-inv4-1},
%
%%%%%%%%%%%%%
\begin{figure}
\begin{center}
\includegraphics*[width=11.5cm]{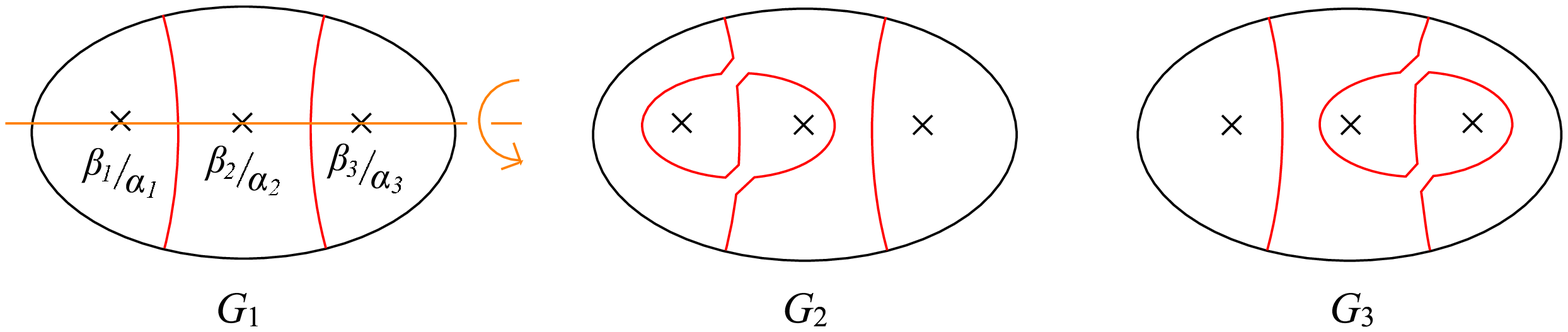}
\end{center}
\caption{}
\label{fig-inv4-1}
\end{figure}
%%%%%%%%%%%%%
%
To be precise, $F\cap M_1$ is isotopic to $f_1(G_i)$ for some $f_1\in\mcg(M_1)$ and for some $i=1,2,3$. 
(We may assume that $f_1|_{M_1}=id$.)
For each $i=1,2,3$,
let $F_i$ be a genus-2 Heegaard surface such that $F_i\cap M_1=G_i$ and $F_i\cap M_2$ is the 2-bridge sphere of $K$.
By \cite[Lemma 6 (2)]{Jan2}, 
any genus-2 Heegaard surface $F$ is isotopic to $f(F_i)$ 
for some integer $n$ and for some $i=1,2,3$ and for some homeomorphism $f\in\mcg_0(M)$ of $M$
which is obtained from some $f_1\in\mcg(M_1)$ by the rule $f|_{M_1}=f_1\in\mcg(M_1)$ and $f|_{M_2}=id$.
Here, $\mcg_0(M)$ denotes the subgroup of $\mcg(M)$ consisting of the elements whose restrictions to $M_2$ are the identity.

\begin{claim}\label{claim-compare}
$\tau_{f(F_i)}=\tau_{F_j}$ if and only if $i=j$ and $f=1$ in $\mcg(M)$.
\end{claim}

\begin{proof}
Put $\tau:=\tau_{F_1}$.
Then the hyper-elliptic involution $\tau_{F_i}$ associated with $F_i$ is 
$\tau$, $x\tau x^{-1}(=x^2\tau)$ and $y\tau y^{-1}(=y^2\tau)$
according as $i=1,2$ and $3$, respectively,
and $\tau_{f(F_i)}$ is $f\tau_{F_i}f^{-1}$.
Recall from \cite[Proof of Theorem 2 (3)]{Jan2} that
$$
\mcg(M)\cong \mcg(M_1)\cong (P_3/\langle (xy)^3\rangle)\rtimes \langle \tau\rangle
< (B_3/\langle (xy)^3\rangle)\rtimes \langle \tau\rangle,
$$
where $P_3$ and $B_3$ are the pure 3-braid group and the 3-braid group, respectively.
Recall from \cite[Claim 1 (2)]{Jan2} that the centralizer $Z(\tau,\mcg(M))$ of $\tau$ is $\{1,\tau\}\cong\Z_2$.
Hence, an element $f\in\mcg_0(M)\subset\mcg(M)$ commutes with $\tau$
only if $f=1$.

Assume that the hyper-elliptic involution $\tau_{f(F_i)}$ 
associated with $f(F_i)$
coincides with $\tau_{F_j}$
for some $i, j \in\{1,2,3\}$.
Since the involutions $\tau_{f(F_i)}$ and $\tau_{F_j}$ are given by
$f(u\tau u^{-1})f^{-1}$ and $v\tau v^{-1}$, respectively,
for some $u,v \in\{1, x, y \}$,
we have 
$$
f(u\tau u^{-1})f^{-1}=v\tau v^{-1},
%\tau(v^{-1}fu)\tau(u^{-1}f^{-1}v)=1.
$$
Thus $v^{-1}fu\in \mcg_0(M)$ commutes with $\tau$ in $\mcg(M)$.
Hence, $f=vu^{-1}$.
Note that the element $vu^{-1}$ is as in Table \ref{table2},
and the only element which belongs to $P_3$ among them is $1$
since any other element changes the order of singular points.
Hence, $f$ must be $1$ and we also have $i=j$.
\begin{table}
\renewcommand{\arraystretch}{1.1} %sŠÔ'²ß
\begin{center}
\begin{tabular}{|c||c|c|c|}
\hline
{\backslashbox{$u$}{$v$}} & {$1$} & {$x$} & {$y$} \\ \hline\hline
{$1$} & {$1$} & {$a^2b$} & {$ba^2$} \\ \hline
{$x$} & {$ba$} & {$1$} & {$ba^2ba$} \\ \hline
{$y$} & {$ab$} & {$a^2bab$} & {$1$} \\ \hline
%{$1$} & {$\{1,b\}$} & {$\{a^2b,a^2\}$} & {$\{ba^2,ba^2b\}$} \\ \hline
%{$x$} & {$\{ba,a\}$} & {$\{1,a^2ba\}$} & {$\{ba^2ba,b\}$} \\ \hline
%{$y$} & {$\{ab,bab\}$} & {$\{a^2bab,b\}$} & {$\{1,ba^2bab\}$} \\ \hline
\end{tabular}
\caption{}
\label{table2}
\end{center}
\end{table}
\end{proof} 
\vspace{2mm}

Hence, 
$M$ admits a unique genus-2 Heegaard surface
whose hyper-elliptic involution is strongly equivalent to $\tau_L$.
By Theorem \ref{prop-hs-3b-0}, this implies that
the 3-bridge sphere in Figure \ref{fig-3bspheres} (5)
is the unique 3-bridge sphere of $L$.
%
%\end{proof}

%%%%%(Proof of Lemma \ref{lem-p0})%%%%%%%%%%%%%%%%%%%%%%%%%%%%%%%%%%%%%%%%%%%%%%%%%%%%%%%%%%%%%%%%%%%%%%%

\section{Proof of Lemma \ref{lem-p0}}\label{sec-pf-p0}

Let $L'$ be the 3-bridge link and $P$ a 3-bridge sphere as in Lemma \ref{lem-p0}.
Then $P$ satisfies the following condition (P0).
\begin{itemize}
\item[(P0)] the pre-image of $P$ 
is a 1-bridge torus of $K$.
\end{itemize}
This condition is equivalent to the following condition (see \cite[Theorem 1.2]{Mor3}).
\begin{itemize}
\item[(P0$'$)] $P$ is a 2-bridge sphere of $L'$,
i.e., $P$ divides $(S^3,L)$ into two 2-string trivial tangles,
$(B^3_1,t_1)$ and $(B^3_2,t_2)$, and moreover,
$(B^3_i,t_i,\gamma\cap B^3_i)$ is as illustrated in Figure \ref{fig-gamma1}
for $i=1,2$.
%
%%%%%%%%%%%%
\begin{figure}
\begin{center}
\includegraphics*[width=2.5cm]{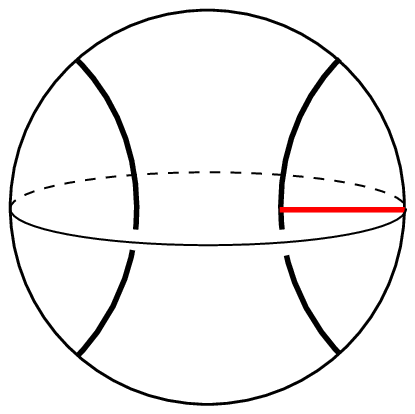}
\end{center}
\caption{}
\label{fig-gamma1}
\end{figure}
%%%%%%%%%%%%
%
\end{itemize}
Let $D$ be the disk bounded by a component of $L'$ 
containing $\gamma$ in it as illustrated in Figure \ref{fig-p}.
Let $K_1$ be the boundary of $D$ and let $K_2$ be the other component of $L'$.
Since $P$ meets each $K_i$ in two points and
since $P$ meets $\gamma$ in a single point,
one of the following conditions holds. 
\begin{itemize}
\item[(i)] $D\cap P$ contains an arc $\delta_1$ properly embedded in $D$
which intersects $\gamma$ transversely in a single point (see Figure \ref{fig-intersection-1} (i)), or
\item[(ii)] $D\cap P$ contains an arc $\delta_2$ and a loop $\delta_3$,
such that $\delta_2$ is disjoint from $\gamma$ and that $\delta_3$ intersects $\gamma$ transversely in a single point 
(see Figure \ref{fig-intersection-1} (ii)).
\end{itemize}
%
%%%%%%%%%%%%
\begin{figure}
\begin{center}
\includegraphics*[width=5.5cm]{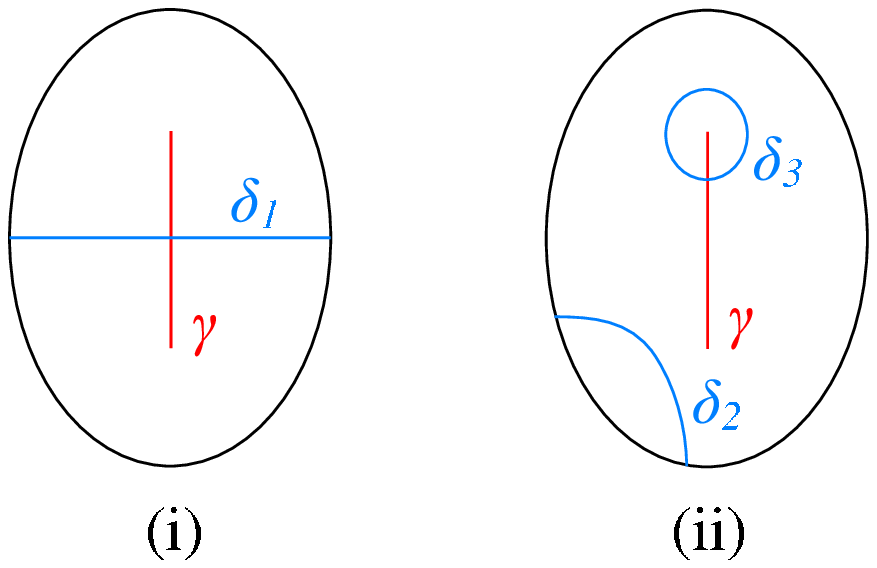}
\end{center}
\caption{}
\label{fig-intersection-1}
\end{figure}
%%%%%%%%%%%%
%

{\bf Case\,(i)} The condition (i) holds.\\[-3mm]

Suppose $D\cap P$ contains a component, $c$, other than $\delta_1$.
Then $c$ is a loop in $D\setminus (\gamma\cup\delta_1)$ 
and hence it bounds a disk, $d_c$, in $D\setminus (\gamma\cup\delta_1)$.
We may assume $c$ is innermost, i.e., $\interior(d_c)\cap P=\emptyset$.
The loop $c$ bounds a disk, $d_c'$, in $P$ such that $|d_c'\cap L'|\leq 2$.
If $|d_c'\cap L'|=0$, then the 2-sphere $d_c'\cup d_c$ bounds a 3-ball in $S^3\setminus L'$.
Thus $P$ can be isotoped so that $c$ is removed from $D\cap P$.
By repeating this deformation, we may assume that 
$D\cap P$ does not contain a loop bounding a disk in $P\setminus L'$.
If $|d_c'\cap L'|=1$, then the 2-sphere $d_c'\cup d_c$ intersects $L'$ in one point, a contradiction.
If $|d_c'\cap L'|=2$, then $c$ is isotopic in $P\setminus L'$ 
to the boundary of a regular neighborhood of $\delta_1$ in $P$,
because the disk $d_c'$ is disjoint from the arc $\delta_1\subset P$. 
This loop represents the commutator $aba^{-1}b^{-1}$
of the two generators $a$ and $b$ of the (2-bridge) link group of $L'$,
where $a$ and $b$ are represented by the meridians of $K_1$ and $K_2$ 
as in Figure \ref{fig-gen}. 
%Here, we take the point of infinity, say $\infty$, 
%as a base point of the link group, and
%the loops $a$, $b$ and $aba^{-1}b^{-1}$, respectively, consist of 
%the oriented triangle from $\infty$ to the tail of the arrows 
%in Figure \ref{fig-gen}
%labeled $a$, $b$ and $aba^{-1}b^{-1}$, respectively, 
%along the arrows to the head,
%then back to $\infty$.
%
%%%%%%%%%%%%
\begin{figure}
\begin{center}
\includegraphics*[width=4.3cm]{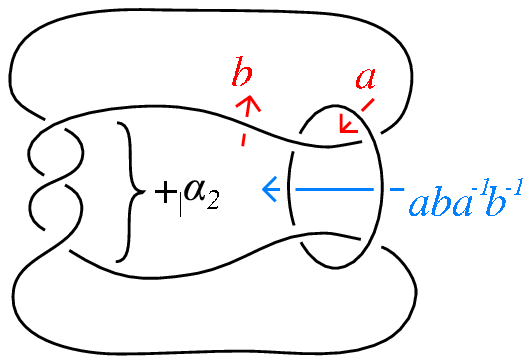}
\end{center}
\caption{}
\label{fig-gen}
\end{figure}
%%%%%%%%%%%%
%
Since the loop bounds the disk $d_c$ in $S^3\setminus L'$, 
we have $aba^{-1}b^{-1}=1$.
This implies that 
the link group is a commutative group,
which is a contradiction.

Hence, we have $P\cap D=\delta_1$. 

Cut $S^3$ along $D$ and close it with two copies of $D$.
Then we have a rational tangle of slope $\pm 1/\alpha$
and the image of $P$ is a disk whose boundary has slope $0$.
Since such a disk is unique up to isotopy fixing the boundary of the tangle, 
$P$ can be isotoped to the 2-sphere $P_0$ in Figure \ref{fig-p} 
by an isotopy fixing $\gamma$.
\vspace{2mm}

{\bf Case\,(ii)} The condition (ii) holds.
\vspace{2mm}

Suppose $D\cap P$ contains a component, $c$, other than $\delta_2\cup \delta_3$.
By an argument similar to that in the previous case, 
we may assume that $c$ dos not bound a disk in $D\setminus (\gamma\cup\delta_2\cup\delta_3)$.
Then $c$ is a separating loop in $D\setminus \gamma$.
Since $c$ is isotopic to $K_1$ in $S^3\setminus K_2$, 
the union $c\cup K_2$ is equivalent to the nontrivial 2-bridge link $L'$.
On the other hand, since the linking number of $c$ and $K_2$ is even,
$c$ bounds a disk in $P\setminus L'$ or 
separates $P\cap K_1$ and $P\cap K_2$. 
In the former case, $c$ bounds a disk in $S^3\setminus L'$, 
which contradicts the fact that $c\cup K_2$ is a nontrivial 2-bridge link.
In the latter case, $c$ is isotopic in $P\setminus L'$
to the boundary of a regular neighborhood of $\delta_2$ in $P$.
Since $\delta_2$ bounds a disk in $D$ with an arc on $\partial D$, 
we see that $c$ is null-homotopic in $S^3\setminus L'$, a contradiction.
Hence, we have $P\cap D=\delta_2\cup \delta_3$. 

Since $P$ satisfies the condition (P0$'$),
there is a height function $h:S^3\rightarrow [-1,1]$ 
such that 
$P_t:=h^{-1}(t)$ satisfies the condition (P0$'$) when $-1<t<1$,
%meets $L'$ and $\gamma$, respectively, 
%in four points and in one point when $-1<t<1$,
and that $P_{\pm 1}$ is an arc meeting $K_i$ $(i=1,2)$ in a single point,
where $K_2\cap \gamma=K_2\cap (P_{+1}\cup P_{-1})$ (see Figure \ref{fig-height}).
%
%%%%%%%%%%%%
\begin{figure}
\begin{center}
\includegraphics*[width=3.5cm]{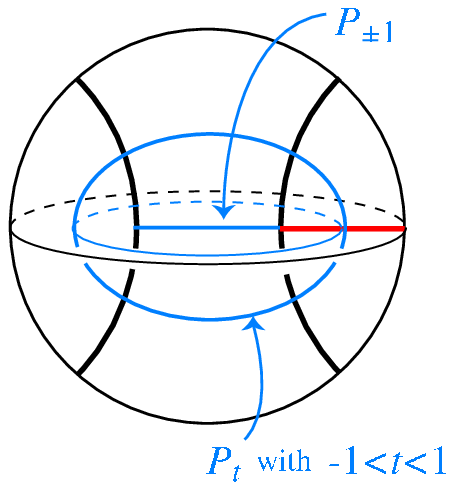}
\end{center}
\caption{}
\label{fig-height}
\end{figure}
%%%%%%%%%%%%
%
Moreover, we may assume that $P_{0}=P$
and that the restriction $g:=h|_{D}$ of $h$ to $D$ has 
at most one non-degenerate singular point
at every level.
Thus, for every singular value $t_0$, 
$g^{-1}({t_0})$ contains a maximal point, a minimal point or a saddle point.
We represent each saddle point in $g^{-1}({t_0})$ by an arc
with endpoints on $g^{-1}({t_0-\varepsilon})$ 
for sufficiently small $\varepsilon >0$, as in Figure \ref{fig-saddle}.
%
%%%%%%%%%%%%
\begin{figure}
\begin{center}
\includegraphics*[width=7.2cm]{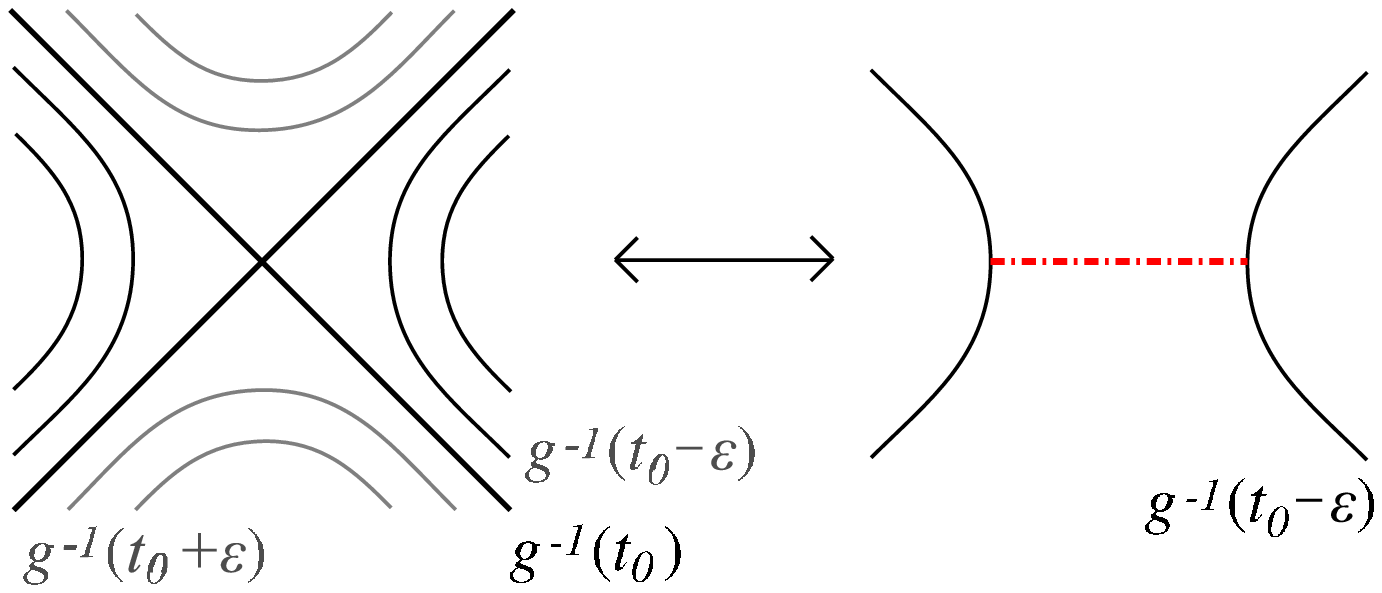}
\end{center}
\caption{}
\label{fig-saddle}
\end{figure}
%%%%%%%%%%%%
%
\begin{lemma}\label{lem-regular}
Let $t$ be a regular value of $g(=h|_{D})$. 
Then $g^{-1}(t)$ does not contain a loop separating $\partial D$ and $\gamma$ in $D$.
\end{lemma}

\begin{proof}
Recall that $P_t:=h^{-1}(t)$ satisfies the condition (P0$'$).
Hence, $D\cap P_t$ $(=g^{-1}(t))$ satisfies the condition (i) or (ii).
In the former case, 
$D\cap P_t$ does not contain a loop separating $\partial D$ and $\gamma$,
since $D\cap P_t$ contains a properly embedded arc in $D$ which intersects $\gamma$ in a single point.
In the latter case, 
we also see that $D\cap P_t$ does not contain a loop separating $\partial D$ and $\gamma$
by applying the argument at the beginning of Case (ii)
to the 3-bridge sphere $P_t$.
\end{proof}
\vspace{2mm}

Let $t_0$ be a singular value of $g$
and $\alpha$ an arc representing a saddle point in $g^{-1}({t_0})$.
Then the arc $\alpha$ is of one of the following three types 
(see Figure \ref{fig-type}):
%
%%%%%%%%%%%%
\begin{figure}
\begin{center}
\includegraphics*[width=8cm]{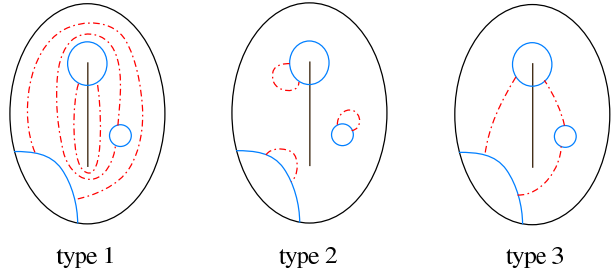}
\end{center}
\caption{The dashed lines give all possible types of an arc representing a saddle point of $g$.}
\label{fig-type}
\end{figure}
%%%%%%%%%%%%
%
\begin{itemize}
\item
$\alpha$ is of {\it type 1} 
if its endpoints are on the same component of $g^{-1}({t_0-\varepsilon})$, 
and $g^{-1}({t_0+\varepsilon})$ contains a loop on $D$ 
which separates $\partial D$ and $\gamma$,
\item
$\alpha$ is of {\it type 2} 
if its endpoints are on the same component of $g^{-1}({t_0-\varepsilon})$, 
and $g^{-1}({t_0+\varepsilon})$ does not contain a loop on $D$ 
which separates $\partial D$ and $\gamma$,
and 
\item
$\alpha$ is of {\it type 3} 
if its endpoints are on different components of $g^{-1}({t_0-\varepsilon})$.
\end{itemize}

By Lemma \ref{lem-regular},
we see that an arc of type 1 does not exist.
%We first show that an arc of type 1 does not exist.
%Suppose, on the contrary, there is an arc of type 1 (see Figure \ref{fig-type}).
%Then $g^{-1}(t_0+\varepsilon)(=P_{t_0+\varepsilon}\cap D)$ contains a loop
%which separates $\partial D$ and $\gamma$ in $D$. 
%By applying the argument at the beginning of Case (ii)
%to the 3-bridge sphere $P_{t_0+\varepsilon}$,
%we see that this cannot occur.
Thus, any arc representing a saddle point of $P_{t_0}$
is of type 2 or of type 3.

Put $X_{s}:=g^{-1}([-1,s])$ for any $s\in [-1,1]$. 
Since $P(=P_{0})$ cuts $D$ into two disks and an annulus,
we may assume that $X_{0}$ is the union of the two disks, 
say $X_{0}^1$ and $X_{0}^2$.
Let $X_{s}^i$ ($s\in(0,1]$) be the component of $X_{s}$ 
which contains $X_{0}^i$ ($i=1,2$).
Since $X_1$ is connected, 
there exists a singular value $s_0\in (0,1)$ 
and a sufficiently small $\varepsilon>0$
such that
$X_{s_0-\varepsilon}^1\neq X_{s_0-\varepsilon}^2$ and
$X_{s_0+\varepsilon}^1 = X_{s_0+\varepsilon}^2$.
Then the arc representing the saddle point in $g^{-1}({s_0})$ 
connects $X_{s_0-\varepsilon}^1$ and $X_{s_0-\varepsilon}^2$.
Note that, at any singular point $s_0'(\neq s_0)$,
$X_{s_0'+\varepsilon}^1\cup X_{s_0'+\varepsilon}^2$ is homeomorphic to
$X_{s_0'-\varepsilon}^1\cup X_{s_0'-\varepsilon}^2$ with some open disks (possibly empty) in it removed.
Hence, $X_{s_0-\varepsilon}^i$ is homeomorphic to
$X_{0}^i$ $(i=1,2)$ with some open disks (possibly empty) in it removed.
Since the arc representing the saddle point at $t=s_0$
connects the outermost components of 
$\partial X_{s_0-\varepsilon}^1$ and $\partial X_{s_0-\varepsilon}^2$,
which are homeomorphic to 
$\partial X_{0}^1$ and $\partial X_{0}^2$, respectively,
$P_{s_0+\varepsilon}$ satisfies the condition (i)
for the previous case (see Figure \ref{fig-intersection-2}).
%
%%%%%%%%%%%%
\begin{figure}
\begin{center}
\includegraphics*[width=9cm]{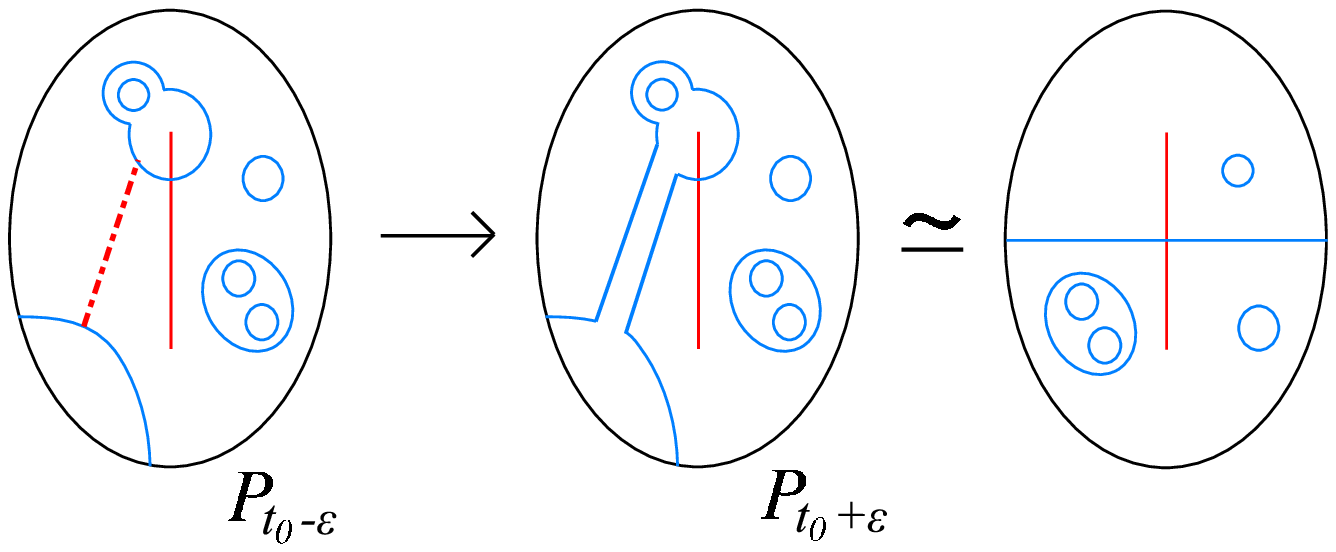}
\end{center}
\caption{}
\label{fig-intersection-2}
\end{figure}
%%%%%%%%%%%%
%

Hence, by the result in Case (i), $P$ can be isotoped to a 2-sphere $P_0$ in Figure \ref{fig-p} 
by an isotopy of $(S^3, L')$ preserving $\gamma$.

%%%%%(3-bridge spheres of the non-simple exceptional link)%%%%%%%%%%%%%%%%%%%%%%%%%%%%%%%%%%%%%%%%%%%%%%%%%%%%%%%%%%%%%%%%%%%%%%%

\section{3-bridge spheres of the non-simple exceptional link}\label{sec-simple-exception}

In this section,
we show that the exceptional 3-bridge arborescent link $L$ 
in Figure \ref{fig-gen-mont} 
admits a unique 3-bridge sphere up to isotopy.

\begin{proposition}\label{lem-unique-sp}
Let $L$ be the link in Figure \ref{fig-gen-mont} 
for some nonzero integer $n$.
Then any 3-bridge sphere of $L$ is isotopic to the 3-bridge sphere $S_0$ 
in Figure \ref{fig-sp2}.
%
%%%%%%%%%%%%%
\begin{figure}
\begin{center}
\includegraphics*[width=3.8cm]{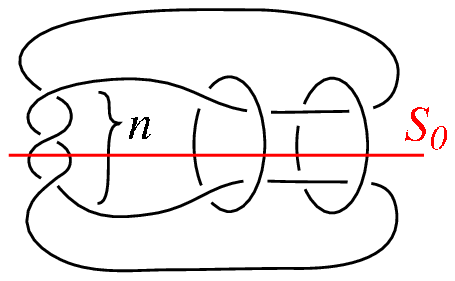}
\end{center}
\caption{}
\label{fig-sp2}
\end{figure}
%%%%%%%%%%%%%
%
\end{proposition}

\begin{remark}
{\rm 
Recall from \cite[Proposition 4]{Jan2} that 
$L$ is equivalent to $L_2((-1/2,1/2),(1/n),$ $(-1/2,1/2))$ or 
$L_1((-1/2,1/2-n),(-1/2,1/2-n))$
according as $|n|>1$ or $|n|=1$.
Moreover, the 3-bridge sphere $S_0$ of $L$ is isotopic to 
the 3-bridge sphere of $L_2((-1/2,1/2),$ $(1/n),(-1/2,1/2))$
in Figure \ref{fig-3bspheres} (4) when $|n|>1$,
and isotopic to the 3-bridge sphere $S_1(=S_2)$
of $L_1((-1/2,1/2-n),(-1/2,1/2-n))$ 
in Figure \ref{fig-3bspheres} (1) or (2) when $|n|=1$.
}
\end{remark}

{\sc Proof of Proposition \ref{lem-unique-sp}.}
%\begin{proof}{\bf of Proposition \ref{lem-unique-sp}}
Let $L$ be the link in Figure \ref{fig-gen-mont}
and $S$ a 3-bridge sphere of $L$.
Let $K_1$ and $K_2$ be the two parallel components of $L$
and $K_3$ the other component.
Note that $K_1\cup K_2$ bounds an annulus, say $A$, in $S^3\setminus K_3$.
Let $D_1$ and $D_2$ be the disjoint disks in $S^3$
bounded by $K_1$ and $K_2$, respectively, 
such that $D_i\cap A=K_i$ and 
$D_i\cap K_3$ consists of two points for each $i=1,2$ 
as illustrated in Figure \ref{fig-sp5}.
%
%%%%%%%%%%%%%
\begin{figure}
\begin{center}
\includegraphics*[width=3.8cm]{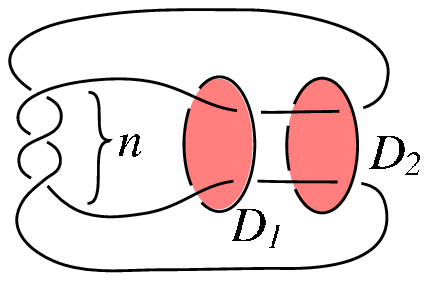}
\end{center}
\caption{}
\label{fig-sp5}
\end{figure}
%%%%%%%%%%%%%
%
Set $P:=A\cup D_1\cup D_2$. 
Then $P$ is a 2-sphere which contains $K_1\cup K_2$
and intersects $K_3$ in four points.
We may assume that $S$ intersects $P$ transversely.
Let $B_1$ and $B_2$ be the 3-balls in $S^3$ bounded by $P$, 
such that $(B_1, B_1\cap K_3)$ and $(B_2, B_2\cap K_3)$
are the tangles as illustrated in Figure \ref{fig-sp10} (1) and (2), respectively.
%
%%%%%%%%%%%%%
\begin{figure}
\begin{center}
\includegraphics*[width=7cm]{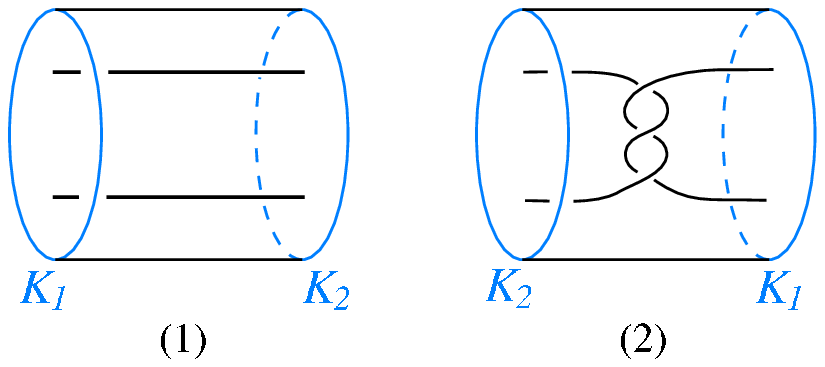}
\end{center}
\caption{}
\label{fig-sp10}
\end{figure}
%%%%%%%%%%%%%
%

Since $L$ consists of 3 components, 
$S$ intersects each component of $L$ in two points.
Hence, one of the following holds.
\begin{itemize}
\item[(A1)] $S\cap A$ contains properly embedded non-separating arcs $\delta_1$ and $\delta_2$ in $A$
as in Figure \ref{fig-sp3} (i), or
\item[(A2)] $S\cap A$ contains properly embedded separating arcs $\delta_3$ and $\delta_4$ in $A$
as in Figure \ref{fig-sp3} (ii).
\end{itemize}
%
%%%%%%%%%%%%%
\begin{figure}
\begin{center}
\includegraphics*[width=6.5cm]{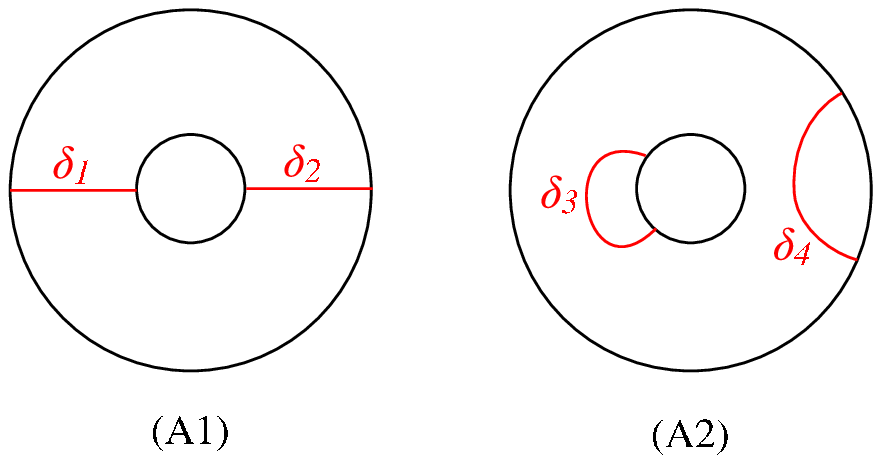}
\end{center}
\caption{}
\label{fig-sp3}
\end{figure}
%%%%%%%%%%%%%
%

On the other hand, 
$S\cap D_i$ $(i=1,2)$ satisfies one of the following conditions.
\begin{itemize}
\item[(D1)] $S\cap D_i$ contains an arc $\varepsilon_1^i$ 
properly embedded in $D_i$ which separates the two points $D_i\cap K_3$.
\item[(D2)] $S\cap D_i$ contains an arc $\varepsilon_2^i$ 
properly embedded in $D_i$ which is parallel to the boundary of $D_i$ 
in $D_i\setminus K_3$.
\end{itemize}

\begin{cacase}
Suppose that the condition (A1) holds.
\end{cacase}

\begin{subcacase}
Suppose that both $S\cap D_1$ and $S\cap D_2$ satisfy the condition (D1).
\end{subcacase}

Let $\gamma_1$ be the loop 
$\delta_1\cup \delta_2\cup \varepsilon_1^1\cup \varepsilon_1^2$.
Then $\gamma_1$ bounds two disks, say $\Delta_1$ and $\Delta_2$,
in $S$.
We can see that $\gamma_1$ is obtained from the loop $\gamma_0$ in Figure \ref{fig-sp9}
by applying (half) Dehn twists along the core loop of $A$.
%
%%%%%%%%%%%%%
\begin{figure}
\begin{center}
\includegraphics*[width=4cm]{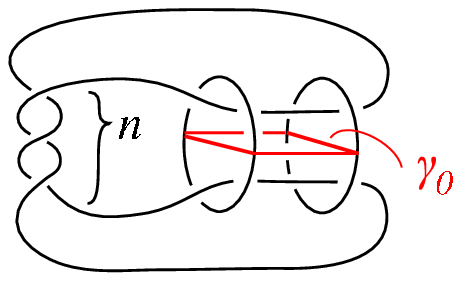}
\end{center}
\caption{}
\label{fig-sp9}
\end{figure}
%%%%%%%%%%%%%
%
Note that the linking number of $\gamma_1$ and $K_3$ is even,
which implies that each $\Delta_i$ intersects $K_3$ in an even number of points.
Since $S$ intersects $K_3$ in two points, 
one of $\Delta_1$ and $\Delta_2$, say $\Delta_1$, is disjoint from $K_3$
and the other meets $K_3$ in two points.

Suppose that $\interior(\Delta_1)\cap P\neq \emptyset$,
and pick a (loop) component, $c_0$, of $\interior(\Delta_1)\cap P$
which is innermost in $\Delta_1$.
Since $\Delta_1$ is disjoint from $K_3$,
the disk $d_1$ bounded by $c_0$ in $\Delta_1$ is also disjoint from $K_3$.
On the other hand, since $c_0$ is disjoint from $\gamma_1$,
$c_0$ bounds a disk, $d_2$ in $P$
intersecting $K_3$ in at most one point.
Hence, $d_2\cap K_3=\emptyset$.
Since $L$ is unsplittable, 
the 2-sphere $d_1\cup d_2$ bounds a 3-ball disjoint from $L$.
Thus we may remove the loop component $c_0$ by an isotopy.
By repeating this, we may assume that $\interior(\Delta_1)\cap P$ is empty.
Hence, $\Delta_1\subset B_1$ or $\Delta_1\subset B_2$. 

Recall that $(B_1,B_1\cap K_3)$ and $(B_2,B_2\cap K_3)$ are rational tangles
of \lq\lq slopes\rq\rq\  $0/1$ and $1/n$, respectively (see Figure \ref{fig-sp10}).
Since $\gamma_1$ is an essential loop on $P\setminus K_3$ 
which bounds a disk $\Delta_1\subset B_1\setminus K_3$,
this implies that $\gamma_1$ is isotopic to $\gamma_0$,
and $\Delta_1$ is isotopic to the disk as in Figure \ref{fig-sp12} (1). 
%
%%%%%%%%%%%%%
\begin{figure}
\begin{center}
\includegraphics*[width=7cm]{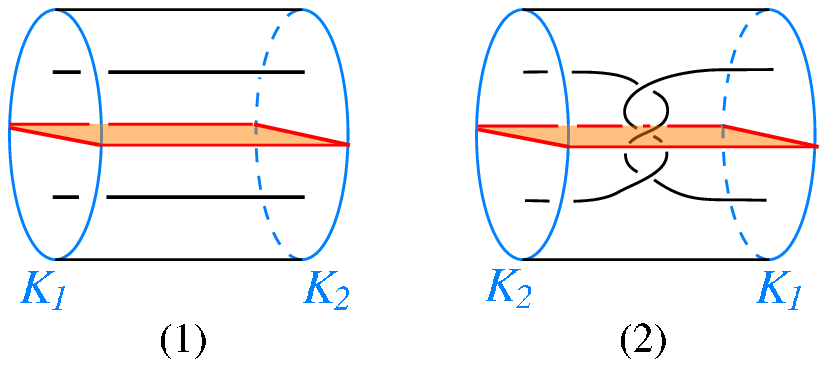}
\end{center}
\caption{}
\label{fig-sp12}
\end{figure}
%%%%%%%%%%%%%
%

Note that $\interior(\Delta_2)\cap L=\interior(\Delta_2)\cap K_3$ and it consists of two points.
Let $c$ be a component of $\interior(\Delta_2)\cap P$.
Then one of the following holds (see Figure \ref{fig-delta2} (1)).
%
%%%%%%%%%%%%%
\begin{figure}
\begin{center}
\includegraphics*[width=7cm]{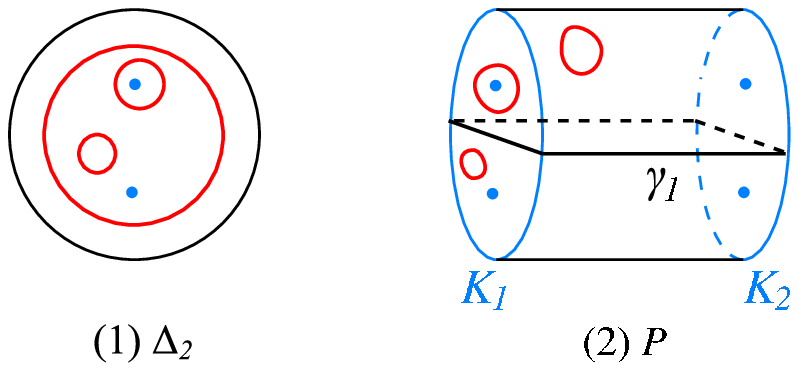}
\end{center}
\caption{}
\label{fig-delta2}
\end{figure}
%%%%%%%%%%%%%
%
\begin{itemize}
\item[(i)] $c$ bounds a disk in $\interior(\Delta_2)\setminus K_3$,
\item[(ii)] $c$ bounds a disk in $\interior(\Delta_2)$ which meets $K_3$ in a single point,
\item[(iii)] $c$ is parallel to $\gamma_1=\partial\Delta_2$ in $\Delta_2\setminus K_3$.
\end{itemize}
On the other hand, $c$ is disjoint from $K_1\cup K_2\cap \gamma_1$,
and hence bounds a disk in $P\setminus(K_1\cup K_2\cup \gamma_1)$
which meets $K_3$ in at most one point (see Figure \ref{fig-delta2} (2)).

Let $c_1$ be a loop satisfying the condition (i) which is innermost in $\Delta_2$.
Then $c_1$ must bound a disk also in $P\setminus(L\cup \gamma_1)$,
and hence we can eliminate $c_1$ from $\interior(\Delta_2)\cap P$
by using the 3-ball bounded by the union of the two disks bounded by $c_1$.
In this way, we can eliminate all loops satisfying the condition (i).

Let $c_1$ be a loop satisfying the condition (ii) which is innermost in $\Delta_2$.
Then $c_2$ bounds a disk in $\interior(\Delta_2)$ which meets $K_3$ in a single point,
and hence it also bounds a disk in $P\setminus(K_1\cup K_2\cup \gamma_1)$
which meets $K_3$ in one point.
The union of the two disks is a 2-sphere in $S^3$ which meets $L$ in two points.
Since $L$ is prime, the 2-sphere bounds a 3-ball in $S^3$ which meets $L$ in a trivially embedded arc.
Hence, we can eliminate $c_2$ from $\interior(\Delta_2)\cap P$, and
we can eliminate all loops satisfying the condition (ii) similarly.

Let $c_1$ be a loop satisfying the condition (iii).
Then $c_3$ is homotopic to the loop $\gamma_1'$ in $S^3\setminus L$
as illustrated in Figure \ref{fig-c2}.
%
%%%%%%%%%%%%%
\begin{figure}
\begin{center}
\includegraphics*[width=3.7cm]{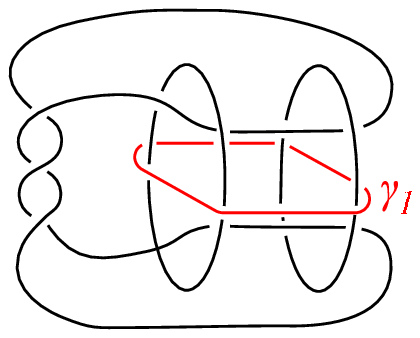}
\end{center}
\caption{}
\label{fig-c2}
\end{figure}
%%%%%%%%%%%%%
%
By an argument similar to that in Case (ii) of the proof of Lemma \ref{lem-p0},
we can see that $c_3$ is not null-homotopic in $S^3\setminus L$.
On the other hand, $c_3$ bounds a disk in $P\setminus L(\subset S^3\setminus L)$, a contradiction.
Hence, we may assume that $\interior(\Delta_2)\cap P$ is empty,
that is, $\Delta_2\subset B_2$.
Since $\Delta_2$ meets $K_3$ in two points,
we see that $\Delta_2$ is isotopic to the disk as in Figure \ref{fig-sp12} (2).

Therefore, $S$ is isotopic to $S_0$ in Figure \ref{fig-sp2}.

\begin{subcacase}
Suppose that $S\cap D_1$ and $S\cap D_2$ satisfy (D1) and (D2), respectively.
\end{subcacase}

By an argument similar to that in the previous case, 
together with the following sublemma, 
we can see that $S\cap P$ is isotopic to the loop $\gamma_3$
as in Figure \ref{fig-sp13}
and that $S$ can be obtained by gluing the two disks 
in Figure \ref{fig-sp13} (1) and (2).
\begin{sublemma}\label{sublemma}
The intersection $S\cap P$ does not contain a loop parallel to $K_1$ (or $K_2$) in $P\setminus L$.
\end{sublemma}
%
%%%%%%%%%%%%%
\begin{figure}
\begin{center}
\includegraphics*[width=7cm]{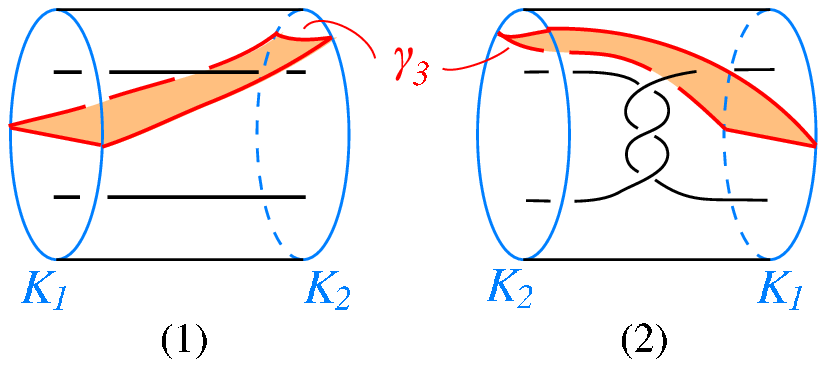}
\end{center}
\caption{}
\label{fig-sp13}
\end{figure}
%%%%%%%%%%%%%
%
\begin{proof}
Assume on the contrary that $S\cap P$ does not contain a loop $c$ parallel to $K_1$ (or $K_2$).
Then the union of $c$ and $K_3$ is equivalent to the sublink $K_1\cap K_3$ of $L$.
Since $c\cup K_3$ is a nontrivial 2-bridge link with linking number $0$ or $\pm 2$,
any disk bounded by $c$ meets $K_3$ in at least two points.
Note that $c$ cuts $S$ into two disks.
Then the above observation implies that $S$ meets $K$ in at least four points, a contradiction.
\qed
%\end{proof}

Hence, $S$ can be isotoped to $S_0$.

\begin{subcacase}
Suppose that both $S\cap D_1$ and $S\cap D_2$ satisfy (D2).
\end{subcacase}

Let $h':S^3\rightarrow [-1,1]$ be a height function 
such that $S_t:=h'^{-1}(t)$ is a 2-sphere which meets $K_i$ in two points for each $i=1,2,3$ when $-1<t<1$,
$S_0=S$ in particular,
and $S_{\pm 1}$ is an arc which meets $K_i$ in one point for each $i=1,2,3$.
By applying an argument similar to that for the height function $g=h|_D$ 
in Case (ii) in the proof of Lemma \ref{lem-p0} to $h'|_{D_1}$ (or $h'|_{D_2}$),
we can see that there exists $t_0\in (-1,1)$ such that 
$S_{t_0}$, isotopic to $S$, satisfies 
the assumption in Case 1.1 or Case 1.2.

\begin{cacase}
Suppose that the condition (A2) holds.
\end{cacase}

Note that every loop component of $A\cap S$
bounds a disk in $A$ or is isotopic to the core loop of $A$.
By Sublemma \ref{sublemma}, any loop component of $A\cap S$ cannot be isotopic to the core loop of $A$.
Hence, $A\cap S$ consists of only loop components bounding a disk in $A$.
By using an argument similar to that for the \lq\lq height function\rq\rq\
in the previous case,
we see that this case can be reduced to Case 1.

This completes the proof of Proposition \ref{lem-unique-sp}.
\end{proof}

%%%%%(Classification of 3-bridge spheres)%%%%%%%%%%%%%%%%%%%%%%%%%%%%%%%%%%%%%%%%%%%%%%%%%%%%%%%%%%%%%%%%%%%%%%%

\section{Classification of 3-bridge spheres}\label{classify}

In this section, we prove Theorem \ref{thm-distinction}. 

Let $L$ be a 3-bridge arborescent link and 
suppose that $L$ is not a Montesinos link.
By Theorem \ref{thm-main-2} (ii), 
$L$ admits only one 3-bridge sphere up to isotopy if $L\not\in \LL_1$.
Hence, we focus on the links in $\LL_1$.
Recall that $\LL_1$ is the family of 3-bridge arborescent links in Figure \ref{fig-3blinks} (1).
Then, the double branched covering of $S^3$ 
branched along a link in $\LL_1$
is a 3-manifold obtained from two Seifert fibered spaces 
$D(\beta_i/\alpha_i, \beta_i'/\alpha_i')$ $(i=1,2)$ over a disk
by gluing their boundaries 
so that a regular fiber and a horizontal loop of $M_1$
are identified with a horizontal loop and a regular fiber of $M_2$,
respectively.

The proof of Theorem \ref{thm-distinction} is based on the following fact 
(see \cite{Boi}). 

\begin{proposition}\label{nielsen}
Let $V_1\cup_F V_2$ and $W_1\cup_G W_2$ be genus-$g$ Heegaard splittings of a $3$-manifold $M$
such that $F$ and $G$ are isotopic.
Assume that the isotopy carries $V_1$ to $W_i$ for $i=1$ or $2$.
Then the generating system $\{x_1, x_2,\dots, x_g\}$ of $\pi_1(M)$ determined by that of $\pi_1(V_1)$
is Nielsen equivalent to
the generating system $\{y_1, y_2,\dots, y_g\}$ of $\pi_1(M)$ determined by that of $\pi_1(W_i)$.
In particular, if $g=2$ then the commutator $[x_1, x_2]$ is conjugate to $[y_1, y_2]^{\pm 1}$.
\end{proposition}

By using this proposition, 
we distinguish, up to isotopy, the Heegaard surfaces which appear in the proof of Theorem \ref{thm-main-2}.

We need the following lemma to solve the conjugacy problems 
that appear in Lemmas \ref{distinction-hs-1} and \ref{distinction-hs-2}.

\begin{lemma}\label{lem-word-jan}
Let $M=D(\beta_1/\alpha_1, \beta_2/\alpha_2)$ be a Seifert fibered space over a disk 
with two exceptional fibers ($\alpha_i>1$), 
then $\pi_1(M)$ has a group presentation
$$
\pi_1(M)\cong \langle c_1, c_2, h 
\mid [c_j,h], c_j^{\alpha_j}h^{\beta_j} (j=1,2)
\rangle ,
$$
where $\pi_1(\partial M)=\langle c_1c_2, h\rangle$.
For $i=1,2$, 
let $\eta_i$ be the element of $\pi_1(M)$ represented by the exceptional fiber of $M$ 
with Seifert index $\beta_i/\alpha_i$, namely, 
$\eta_i=c_i^{\gamma_i}h^{\delta_i}$
for some $\gamma_i$ and $\delta_i$ such that $\alpha_i\delta_i-\beta_i\gamma_i=1$. 

For integers a, b, c and d, let 
$$
w(a,b,c,d)= \{(c_1c_2)^ah^b\} \eta_1 \{(c_1c_2)^ch^d\}\in \pi_1(M). 
$$
Then the followings are the only solutions of the equation 
$w(a,b,c,d)=\eta_1^{\pm 1}\ or\ \eta_2^{\pm 1}${\rm :}
\begin{eqnarray*}
\begin{array}{rl}
{\rm (i)} & w(0,b,0,-b)=\eta_1, \\
{\rm (ii)} & w(\pm 1,b,\pm 1,-b-2k_1\pm \beta_2)=\eta_1^{-1}\ 
when\ \beta_1=\pm 1+k_1\alpha_1\ and\ \alpha_2=2,\\
{\rm (iii)} & w(-1,b,0,-b-k_1-k_2)=\eta_2^{\pm 1}\ 
when\ \beta_1=-1+k_1\alpha_1\ and\ \beta_2=\pm 1+k_2\alpha_2,\\
{\rm (iv)} & w(0,b,1,-b+k_1+k_2)=\eta_2^{\pm 1}\ 
when\ \beta_1=1+k_1\alpha_1\ and\ \beta_2=\mp 1+k_2\alpha_2,
\end{array}
\end{eqnarray*}
where $k_i$ is an integer $(i=1,2)$.
\end{lemma}

Let $A*B$ be the free product of two nontrivial groups $A$ and $B$. 
A word $w=g_1g_2\cdots g_n \in$ $A*B$ $(n\geq 0)$ is said to be of {\it normal form} 
if (i) $g_i\neq 1$, 
(ii) $g_i\in A$ or $g_i\in B$ ($1\leq i\leq n$) and 
(iii) $g_i\in A$ iff $g_{i+1}\in B$ ($1\leq i\leq n-1$). 
Here, $n$ is called the {\it length} of $w \in A*B$ 
and denoted by $|w|$. 
Then, if a word $w=g_1g_2\cdots g_n \in A*B$ is of normal form and $n>1$, 
then $w\neq 1$ in $A*B$ (see, for example, \cite[Ch. IV Theorem 1.2]{Lyn}).

To prove Lemma \ref{lem-word-jan}, 
we improve the argument in \cite[Lemma 4.5]{Mor2} and \cite[Lemma 4.3]{Jan1}
which was used to solve certain word problems in the torus knot group.\\

{\sc Proof of Lemma \ref{lem-word-jan}.}
%\begin{proof}{\bf of Lemma \ref{lem-word-jan}} 
%
We describe only the proof for the case $a\leq -1$. 
The other cases can be treated by similar arguments.

Consider the quotient group
\begin{eqnarray*}
\pi_1(M)/\langle h\rangle \cong 
\langle c_1 \mid c_1^{\alpha_1}=1 \rangle * 
\langle c_2 \mid c_2^{\alpha_2}=1 \rangle.
\end{eqnarray*}
Suppose $w(a,b,c,d)=\eta_1^{\pm 1}$ or $\eta_2^{\pm 1}$ in $\pi_1(M)$. 
Then we have $\hat{w}(a,c)=c_1^{\pm \gamma_1}$ or $c_2^{\pm \gamma_2}$ in $\pi_1(M)/\langle h\rangle$, 
where $\hat{w}(a,c)=(c_1c_2)^ac_1^{\gamma_1}(c_1c_2)^c$ is the element of the quotient group represented by $w(a,b,c,d)$. 

Suppose $c\geq 1$.
Then we have $$\hat{w}(a,c)=(c_2^{-1}c_1^{-1})^{|a|-1}c_2^{-1}c_1^{\gamma_1}c_2(c_1c_2)^{c-1}.$$ 
Thus $|\hat{w}(a,c)|>1$ and hence the equation has no solution. 

Suppose $c=0$.
Then we have $$\hat{w}(a,0)=(c_2^{-1}c_1^{-1})^{|a|-1}c_2^{-1}c_1^{\gamma_1-1}.$$
So $|\hat{w}(a,0)|=1$ if and only if 
$a=-1$ and $\gamma_1\equiv 1\pmod{\alpha_1}$, namely, $\gamma_1=1+k_1'\alpha_1$ for some integer $k_1'$, 
since the order of $c_1$ in $\pi_1(M)/\langle h \rangle$ is $\alpha_1$. 
Recall that $\alpha_1\delta_1-\beta_1\gamma_1=1$. 
Thus we have $\beta_1\gamma_1\equiv -1\pmod{\alpha_1}$, which implies that
$\beta_1\equiv -1$ (mod $\alpha_1$), namely, $\beta_1=-1+k_1\alpha_1$ for some integer $k_1$. 
In this case, $\hat{w}(-1,0)=c_2^{-1}$, 
and this implies 
$\pm \gamma_2\equiv -1\pmod{\alpha_2}$ and hence
$\gamma_2=\mp 1+k_2'\alpha_2$ for some integer $k_2'$ and 
$\beta_2=\pm 1+k_2\alpha_2$ for some integer $k_2$. 
Moreover, we have $w(-1,b,0,d)=\eta_2^{\pm 1}$,
which in turn implies
$$
c_2^{\mp \gamma_2-1}c_1^{\gamma_1-1}h^{b+d+\delta_1\mp\delta_2}=1
$$
in $\pi_1(M)$. 
Since $c_1^{\gamma_1-1}=c_1^{\alpha_1k_1'}=h^{-\beta_1k_1'}$ and 
$c_2^{\pm \gamma_2-1}=c_2^{\pm \alpha_2k_2'}=h^{\mp \beta_2k_2'}$, 
we obtain 
$$
h^{\pm \beta_2k_2'-\beta_1k_1'+b+d+\delta_1\mp\delta_2}=1,
$$
and hence 
$$
d=-b-\delta_1\pm \delta_2+\beta_1k_1'\mp \beta_2k_2'.
$$
Since $1=\alpha_1\delta_1-\beta_1\gamma_1
=\alpha_1\delta_1-\beta_1(1+\alpha_1k_1')=\alpha_1(\delta_1-\beta_1k_1')-(-1+\alpha_1k_1)$,
we have $\delta_1-\beta_1k_1'=k_1$.
Similarly, we have $\delta_2-\beta_2k_2'=\mp k_2$.
This implies that 
$$
d=-b-k_1-k_2.
$$
Thus we obtain the solution (iii).

Suppose $c\leq -1$.
Then we have $$\hat{w}(a,0)=(c_2^{-1}c_1^{-1})^{|a|-1}c_2^{-1}c_1^{\gamma_1-1}c_2^{-1}c_1^{-1}(c_2^{-1}c_1^{-1})^{|c|-1}.$$ 
So $|\hat{w}(a,0)|=1$ if and only if 
$a=c=-1$, $\gamma_1\equiv 1\pmod{\alpha_1}$, namely, $\gamma_1=1+k_1'\alpha_1$ for some integer $k_1'$, 
and $-2\equiv 0\pmod{\alpha_2}$.
Since $\gamma_1\equiv -\beta_1$ (mod $\alpha_1$), 
we have $\beta_1\equiv -1$ (mod $\alpha_1$), namely, $\beta_1=-1+k_1\alpha_1$ for some integer $k_1$, 
and we also have $\alpha_2=2$. 
In this case, $\hat{w}(-1,-1)=c_1^{-1}$,
and this implies $w(-1,b,-1,d)=\eta_1^{-1}$ and hence
$$
c_2^{-1}c_1^{\gamma_1-1}c_2^{-1}c_1^{\gamma_1-1}h^{b+d+2\delta_1}=1
$$
in $\pi_1(M)$. 
Since $c_1^{\gamma_1-1}=c_1^{\alpha_1k_1'}=h^{-\beta_1k_1'}$ and 
$c_2^{-2}=c_2^{-\alpha_2}=h^{\beta_2}$, 
we obtain 
$$
h^{\beta_2-2\beta_1k_1'+b+d+2\delta_1}=1,
$$
and hence 
$$
d=-b-2\delta_1+2\beta_1k_1'-\beta_2.
$$
Since $\delta_1-\beta_1k_1'=k_1$,
we have
$$
d=-b-2k_1-\beta_2.
$$
Thus we obtain the solution (ii).
\qed
%\end{proof}
\vspace{2mm}

The following lemma can be proved similarly.

\begin{lemma}\label{sublemma}
Let $M=D(\beta_1/\alpha_1, \beta_2, \alpha_2)$, $\eta_i$ ($i=1,2)$ and
$w(a,b,c,d)$ as in Lemma \ref{lem-word-jan}.
Then the equation $w(a,b,c,d)=1$ has no solutions
(i.e., an exceptional fiber of $M$ is not homotopic to a loop on $\partial M$).
\end{lemma}

\begin{lemma}\label{sublemma2}
An unknotting tunnel $\tau$ of a nontrivial knot $K$ in $S^3$ 
is not homotopic to an arc on $\partial E(K)$.
\end{lemma}

\begin{proof}
If $\tau$ is homotopic to an arc on $\partial E(K)$,
then it follows that the knot group $\pi_1(E(K))$ 
is generated by the image of $\pi_1(\partial E(K))$, 
and hence $\pi_1(E(K))$ is abelian. 
This contradicts the assumption that $K$ is nontrivial.
\end{proof}
\vspace{2mm}

Recall from Theorem \ref{thm-main-2} (i) that a link $L\in \LL_1$ admits 
at most four 3-bridge spheres $S_1$, $S_2$, $S_3$ and $S_4$ 
in Figure \ref{fig-3bspheres}
up to isotopy.
In the remainder of this section, 
let $F_i$ $(i=1,2,3,4)$ be the pre-image of $S_i$ 
in $M_2(L)$.

The following lemma gives a necessary condition for $F_1$ and $F_2$
to be isotopic.

\begin{lemma}\label{distinction-hs-1}
Let $M$ be a manifold which belongs to M(1-a) in \cite[Theorem 5]{Jan2}, 
that is, $M$ is obtained from $M_1=D(\beta_1/\alpha_1, \beta_1'/\alpha_1')$ 
and $M_2=D(\beta_2/\alpha_2, \beta_2'/\alpha_2')$ by gluing their boundaries
so that a regular fiber of $M_1$ is identified with a horizontal loop of $M_2$, 
where $\alpha_i, \alpha_i'>1$ for $i=1,2$.
Let $F_1$ and $F_2$ be the two genus-2 Heegaard surface of $M$
given as above.

Suppose that 
$
(\beta_k/\alpha_k, \beta_k'/\alpha_k')\not\sim
(\varepsilon_k/\alpha_k, \varepsilon_k'/\alpha_k')
$
for each $k=1,2$,
where $\varepsilon_k, \varepsilon_k'\in\{\pm 1\}$.
Then $F_1$ and $F_2$ are not isotopic.
\end{lemma}

\begin{proof}
Let $U_1\cup U_2$ be a decomposition of $M_1$ 
by a saturated annulus 
and let $W_1\cup W_2$ be a one-bridge decomposition of $M_2$.
Put $V_1^1=U_1\cup W_1$, $V_2^1=U_2\cup W_2$, 
$V_1^2=U_1\cup W_2$ and $V_2^2=U_2\cup W_1$.
Then we may assume that
$F_i=\partial V_1^i=\partial V_2^i$ 
(see Figure \ref{fig-hs} (F1), \cite{Mor} and \cite[Case 1 in Section 7]{Jan2}).

We describe the generating system of the fundamental group $\pi_1(M)$ of $M$
arising from each handlebody, $V_j^i$ $(i,j\in\{1,2\})$.
Pick a base point $x_0$ for the fundamental group of $M$ on $T\cap F_1\cap F_2$,
where $T:=\partial M_1=\partial M_2$.
Let $u_i$ and $v_i$ be exceptional fibers of $M_i$ 
whose Seifert indices are $\beta_i/\alpha_i$ and $\beta_i'/\alpha_i'$, 
respectively.
Connect these loops to $x_0$ by arcs in $M_i$ 
which does not meet $F_i$. 
We denote the generators of $\pi_1(M, x_0)$ 
obtained from $u_i$, $v_i$ with the arcs above
by $u_i$, $v_i$ again.
Then we have generating systems 
$\{u_2, u_1\}$ for $\pi_1(V_1^1)$, $\{v_2, v_1\}$ for $\pi_1(V_2^1)$,
$\{u_2, v_1\}$ for $\pi_1(V_1^2)$ and $\{v_2, u_1\}$ for $\pi_1(V_2^2)$.
This can be seen by using the fact that
the 1-bridge decomposition $W_1\cup W_2$ of $M_2$ can be chosen so that
$W_1$ is the regular neighborhood in $M_2$ of the graph obtained by connecting 
a horizontal loop and the exceptional fiber of $M_2$ of index $\beta_2/\alpha_2$
(see Figure \ref{fig-gen-sys} and \cite[Figure 18]{Jan2}).
%
%%%%%%%%%%%%%
\begin{figure}
\begin{center}
\includegraphics*[width=4cm]{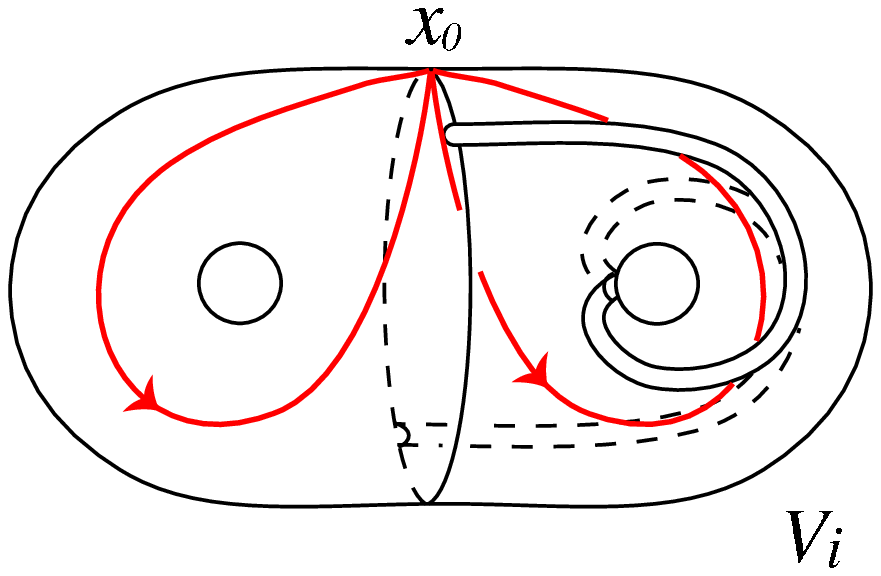}
\end{center}
\caption{}
\label{fig-gen-sys}
\end{figure}
%%%%%%%%%%%%%
%

Suppose that 
$
(\beta_k/\alpha_k, \beta_k'/\alpha_k')\not\sim
(\varepsilon_k/\alpha_k, \varepsilon_k'/\alpha_k')
$
for both $k=1,2$,
where $\varepsilon_k, \varepsilon_k'\in\{\pm 1\}$.
We prove $F_1$ and $F_2$ are not isotopic.

Assume, on the contrary, that $F_1$ and $F_2$ are isotopic.
Then, by Proposition \ref{nielsen}, 
$[u_2, u_1]$ is conjugate to $[u_2, v_1]^{\pm 1}$ or $[v_2, u_1]^{\pm 1}$. 

Assume that $[u_2, u_1]$ is conjugate to $[u_2, v_1]$.
Let $f_0:S^1\rightarrow M$ and $f_1:S^1\rightarrow M$
be maps representing the elements $[u_2, u_1]$ and $[u_2, v_1]$, 
respectively.
By the assumption, $f_0$ and $f_1$ are freely homotopic, 
and hence there is a map
$$
\Psi:S^1(=\R/\Z)\times [0,1]\longrightarrow M
$$
such that $\Psi |_{S^1\times \{0\}}=f_0$ and $\Psi |_{S^1\times\{1\}}=f_1$.
We may assume that $\Psi$ is transverse to $T(=\partial M_1=\partial M_2)$ and 
that $f_i^{-1}(T)$ consists of 4 points $(i=0,1)$. 
Note that the images of the 4 points by $f$ are all equal to the base point $x_0$ of $\pi_1(M)$.
To be precise, we may assume that
\begin{align*}
\begin{array}{rllrll}
\Psi([0,\frac{1}{4}]\times\{0\}) & = & u_2, &
\Psi([0,\frac{1}{4}]\times\{1\}) & = & u_2, \\
\Psi([\frac{1}{4},\frac{1}{2}]\times\{0\}) & = & u_1, &
\Psi([\frac{1}{4},\frac{1}{2}]\times\{1\}) & = & v_1, \\
\Psi([\frac{1}{2},\frac{3}{4}]\times\{0\}) & = & u_2^{-1}, &
\Psi([\frac{1}{2},\frac{3}{4}]\times\{1\}) & = & u_2^{-1}, \\
\Psi([\frac{3}{4},1]\times\{0\}) & = & u_1^{-1}, &
\Psi([\frac{3}{4},1]\times\{1\}) & = & v_1^{-1}.
\end{array}
\end{align*}
Since $\Psi$ is transverse to $T$, 
$\Psi^{-1}(T)$ is a 1-dimensional submanifold of $S^1\times I$.
Since $T$ is incompressible, 
we may further assume that $\Psi^{-1}(T)$ consists of only arcs.
Since $u_i$ and $v_i$ $(i=1,2)$ cannot be homotoped into the boundary 
(see Lemma \ref{sublemma}),
each component of $\Psi^{-1}(T)$ is an arc 
joining $S^1\times\{0\}$ and $S^1\times\{1\}$.
Then, noting the intersection of $\Psi |_{S^1\times\{i\}}$ and $T$,
we see that the map $\Psi$ is as in Figure \ref{fig-word5} (1) or in Figure \ref{fig-word5} (2).
%
%%%%%%%%%
\begin{figure}
\begin{center}
\includegraphics*[width=8.5cm]{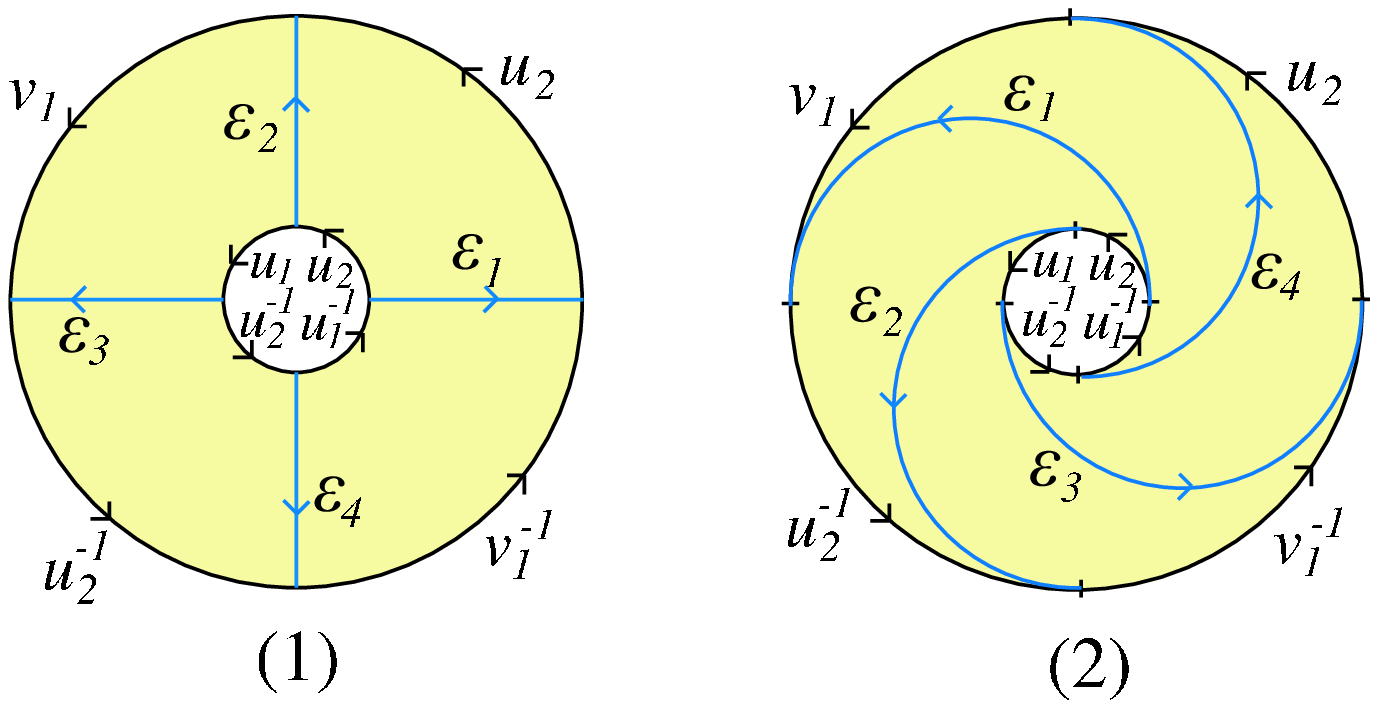}
\end{center}
\caption{}
\label{fig-word5}
\end{figure}
%%%%%%%%%
That is, we may assume 
\begin{eqnarray}\label{equation}
&\left\{
\begin{array}{l}
\Psi^{-1}(M_2) = ([0,\frac{1}{4}]\cup[\frac{1}{2},\frac{3}{4}])\times[0,1],\\
\Psi^{-1}(M_1) = ([\frac{1}{4},\frac{1}{2}]\cup[\frac{3}{4},1])\times[0,1],
\end{array}
\right.
\end{eqnarray} 
or
\begin{eqnarray}\label{equation2}
&\left\{
\begin{array}{l}
\Psi^{-1}(M_2) = \{(t+\frac{1}{2}s,s)\mid 
t\in [0,\frac{1}{4}]\cup[\frac{1}{2},\frac{3}{4}],s\in [0,1]\},\\
\Psi^{-1}(M_1) = \{(t+\frac{1}{2}s,s)\mid 
t\in [\frac{1}{4},\frac{1}{2}]\cup[\frac{3}{4},1],s\in [0,1]\}.
\end{array}
\right.
\end{eqnarray}

Assume that the identity (\ref{equation}) holds. 
Let $\varepsilon_i$ be the elements of $\pi_1(\partial M_1,x_0)$ 
represented by $\Psi\mid_{\{\frac{i-1}{4}\}\times[0,1]} (i=1,2,3,4)$.\ Then
%
%\begin{equation}
%\begin{array}{rlll}
%\varepsilon_1^{-1}u_2\varepsilon_2 & = & u_2 & \in \pi_1(M_2),\\
%\varepsilon_2^{-1}u_1\varepsilon_3 & = & v_1 & \in \pi_1(M_1),\\
%\varepsilon_3^{-1}u_2^{-1}\varepsilon_4 & = & u_2^{-1} & \in \pi_1(M_2),\\
%\varepsilon_4^{-1}u_1^{-1}\varepsilon_1 & = & v_1^{-1} & \in \pi_1(M_1).
%\end{array}
%\end{equation}

\begin{align}
\varepsilon_1^{-1}u_2\varepsilon_2 & = u_2 \in \pi_1(M_2),\label{eqn1}\\
\varepsilon_2^{-1}u_1\varepsilon_3 & = v_1 \in \pi_1(M_1),\label{eqn2}\\
\varepsilon_3^{-1}u_2^{-1}\varepsilon_4 & = u_2^{-1} \in \pi_1(M_2),\label{eqn3}\\
\varepsilon_4^{-1}u_1^{-1}\varepsilon_1 & = v_1^{-1} \in \pi_1(M_1).\label{eqn4}
\end{align}

By Lemma \ref{lem-word-jan}, 
equations (\ref{eqn1}) and (\ref{eqn3}) have solutions if and only if
\begin{align}
\varepsilon_1=\varepsilon_2&=h_2^{n_1}(=c_1^{n_1})\label{eqn5}\\
\varepsilon_3=\varepsilon_4&=h_2^{n_2}(=c_1^{n_2})\label{eqn6}
\end{align}
for some integers $n_1$ and $n_2$, 
where $h_i$ is a regular fiber of $M_i$ ($i=1,2$) and 
$c_1$ is a horizontal loop of $M_1$,
which is identified with $h_2$.
From equation (\ref{eqn2}), we have
$$
c_1^{-n_1}u_1c_1^{n_2}=v_1.
$$
By Lemma \ref{lem-word-jan} (iii) and (iv), 
this equation has a solution if and only if 
\begin{itemize}
\item[(i)] $\beta_1=-1+k_1\alpha_1, \beta_1'=1+k_1'\alpha_1'$, $k_1+k_1'=0$,
$n_1=1$ and $n_2=0$, or
\item[(ii)] $\beta_1=1+k_1\alpha_1, \beta_1'=-1+k_1'\alpha_1'$, $k_1+k_1'=0$,
$n_1=0$ and $n_2=1$.
\end{itemize}
The first three equalities in (i) (or (ii)) together imply 
$\frac{\beta_1}{\alpha_1}+\frac{\beta_1'}{\alpha_1'}=
\frac{-1}{\alpha_1}+\frac{1}{\alpha_1'}$
(or $\frac{\beta_1}{\alpha_1}+\frac{\beta_1'}{\alpha_1'}=
\frac{1}{\alpha_1}+\frac{-1}{\alpha_1'}$).
Hence, by the hypothesis, 
there do not exist $\{\varepsilon_i\}_{i=1,2,3,4}$ 
which satisfy equalities (\ref{eqn1})$,\dots,$(\ref{eqn4}). 
This is a contradiction. 

We can also lead to a contradiction 
when the equation (\ref{equation2}) holds 
(see Lemma \ref{lem-word-jan}, (ii), (iii) and (iv)). 
Hence $[u_2,u_1]$ and $[u_2,v_1]$ are not conjugate. 

Similarly, it can be proved that $[u_2, u_1]$ is not conjugate to 
$[u_2, v_1]^{-1}$ or $[v_2, u_1]^{\pm 1}$. 
Hence, $F_1$ and $F_2$ are not isotopic.
\end{proof}
\vspace{2mm}

The following lemma says that 
any two of $F_1$, $F_2$, $F_3$ and $F_4$ cannot be isotopic
unless they are $F_1$ and $F_2$.

\begin{lemma}\label{distinction-hs-2}
Let $M$ be a manifold 
which belongs to M(1-a) or M(2-a) in \cite[Theorem 5]{Jan2}, 
that is, $M$ is obtained from $M_1, M_2\in D[2]$
by gluing their boundaries so that a regular fiber of $M_1$ is identified with a horizontal loop of $M_2$.
Let $\{G_1, G_2\}$ be a subset of the set 
$\{F_1,F_2,F_3,F_4\}$ of genus-2 Heegaard surfaces of $M$,
and suppose that $\{G_1, G_2\}\neq\{F_1,F_2\}$.
%$F_1$ and $F_2$ be Heegaard surfaces of $M$ such that
%(i) for each $i=1,2$, $F_i$ belongs to one of the three families F(1), F(2-1) and F(2-2), and 
%(ii) none of the three families contain both $F_1$ and $F_2$.
Then $G_1$ and $G_2$ are not isotopic.
\end{lemma}

\begin{proof}
First, suppose that $G_1=F_1$ and $G_2=F_4$.
Then $M\in$ M(1-a), and hence $M$ is a union of $M_1\in D[2]$ and 
$M_2=E(S(2n+1,1))=D(1/2, -n/(2n+1))$.

Let $V_1^i$ and $V_2^i$ be genus-2 handlebodies in $M$ bounded by $G_i$. 
We decompose the handlebodies into several parts as follows (see Figure \ref{fig-hs} (F1) and (F2)).
Put $U_j:=V_j^1\cap M_1$ and $W_j:=V_j^1\cap M_2$,
then $U_1\cup U_2$ gives a decomposition of $M_1$ by a saturated annulus and 
$W_1\cup W_2$ gives the one-bridge presentation of $M_2$. 
Note that either $V_1^2$ or $V_2^2$, say $V_2^2$, 
is separated into three components by $T:=\partial M_1=\partial M_2$.
We put $V_1^2:=U_3\cup R$ and $V_2^2:=W_3\cup U_4\cup W_4$, 
where $W_3\cup R\cup W_3$ gives a decomposition of $M_1$ by two parallel saturated annuli 
and $U_3\cup U_4$ gives the two-bridge presentation of $M_2$.

We describe the generating system of 
the $\pi_1(M)$ arising from each handlebody, $V_j^i$ $(i,j\in\{1,2\})$.
Pick a base point $x_0$ for the fundamental group of $M$ on $T\cap G_1$.

Let $u_i$ and $v_i$ be the generators of $\pi_1(M)$ 
obtained from exceptional fibers of $M_i$ 
as in the proof of Lemma \ref{distinction-hs-1}.
Then the generating systems for $\pi_1(V_1^1,x_0)$ and $\pi_1(V_2^1,x_0)$ are
equal to either (i) $\{u_2, u_1\}$ and $\{v_2, v_1\}$, or (ii) $\{u_2, v_1\}$ and $\{v_2, u_1\}$.
%We prove the lemma for the case (i). 
%The other case can be treated similarly.

Pick a point $x_1\in W_3\cap U_4$ and $x_2$ on $U_4\cap W_4\cap G_2$.
The generating system of $\pi_1(V_1^2, x_1)$ is $\{\tau_2\tau_1, h_1\}$ 
and the generating system of $\pi_1(V_2^2, x_1)$ is 
$\{u_1, \tau_2v_1'\tau_2^{-1}\}$, 
where $\tau_2$ is an arc on $U_3\cap U_4$ joining $x_1$ to $x_2$,
$\tau_1$ is an arc in $R$ joining $x_2$ to $x_1$, 
$h_1$ is a regular fiber of $M_1$ 
and $v_1'$ is an element of $\pi_1(V_2^2, x_2)$ 
obtained from $v_1$ by taking conjugation by an arc $\tau$ on $T$
joining $x_2$ to $x_1$.
We abuse notation to denote the loops $\tau^{-1}\tau_1$ and $\tau_2\tau$
by the symbols $\tau_1$ and $\tau_2$ again. 
Then the generating systems of $\pi_1(M, x_1)$ arising from $V_1^2$ and $V_2^2$
are $\{\tau_2\tau_1, h_1\}$ and $\{u_1, \tau_2v_1\tau_2^{-1}\}$, respectively.

Suppose that $G_1$ and $G_2$ are isotopic.
Then, by Proposition \ref{nielsen},
$[u_1, \tau_2v_1\tau_2^{-1}]$ is conjugate to 
$[u_2, u_1]^{\pm 1}$ or $[v_2, v_1]^{\pm 1}$.
In order to show that this is impossible,
recall that $\pi_1(M)$ is the free product of $\pi_1(M_1)$ and $\pi_1(M_2)$
with amalgamated subgroup $\pi_1(T)$.
Thus the length of each word of $\pi_1(M)$ with respect to this structure is defined.
By using Lemma \ref{sublemma},
we can see that,
for each of $[u_2, u_1]^{\pm 1}$ and $[v_2, v_1]^{\pm 1}$,
the minimal length of words conjugate to it is 4.
We can also see by using Lemmas \ref{sublemma} and \ref{sublemma2}
that the minimal length of words conjugate to $[u_1, \tau_2v_1\tau_2^{-1}]$ is 8.
Hence, 
%the length of $[u_1, \tau_2v_1\tau_2^{-1}]$
%is different from the length of $[u_2, u_1]^{\pm 1}$, $[v_2, v_1]^{\pm 1}$, $[u_2, v_1]^{\pm 1}$ or $[v_2, u_1]^{\pm 1}$.
%This implies that 
$[u_1, \tau_2v_1\tau_2^{-1}]$ is not conjugate to 
$[u_2, u_1]^{\pm 1}$ or $[v_2, v_1]^{\pm 1}$.
Hence, $G_1$ and $G_2$ are not isotopic.

Similarly, it can be proved that 
$F_1$ or $F_2$ cannot be isotopic to $F_3$ or $F_4$.
Moreover, by similar arguments,
one can also prove that 
$F_3$ and $F_4$ are not isotopic.
\end{proof}
\vspace{2mm}

{\sc Proof of Theorem \ref{thm-distinction}.}
%\begin{proof}{\bf of Theorem \ref{thm-distinction}}
Suppose that 
$
(\beta_k/\alpha_k, \beta_k'/\alpha_k')\sim
(\varepsilon_k/\alpha_k, \varepsilon_k'/\alpha_k')
$
for some $k=1,2$,
where $\varepsilon_k, \varepsilon_k'\in\{\pm 1\}$.
Then the two 3-bridge spheres $S_1$ and $S_2$ for $L$
are isotopic 
by an isotopy illustrated in Figure \ref{fig-isotopy}.
%
%%%%%%%%%
\begin{figure}
\begin{center}
\includegraphics*[width=10cm]{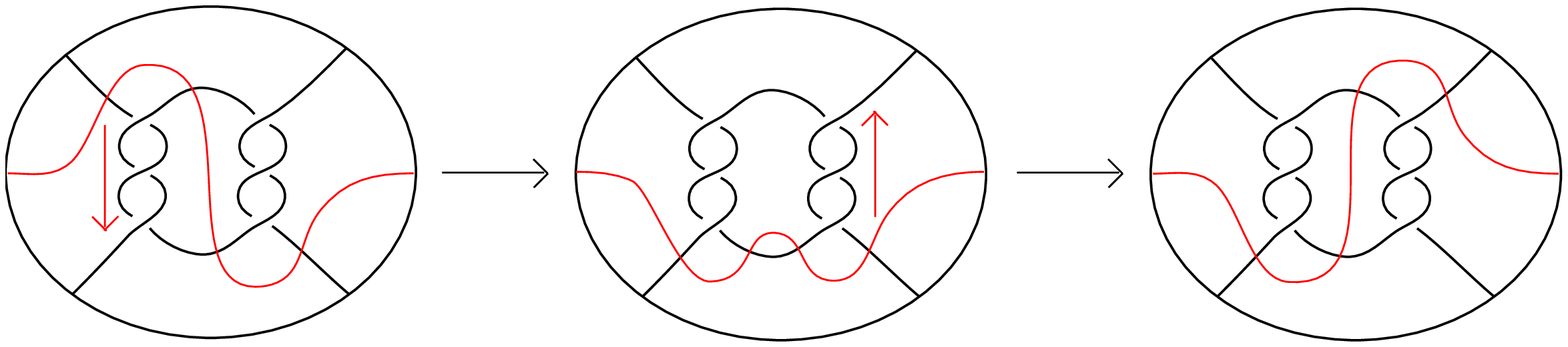}
\end{center}
\caption{}
\label{fig-isotopy}
\end{figure}
%%%%%%%%%
%
If
$
(\beta_k/\alpha_k, \beta_k'/\alpha_k')\not\sim
(\varepsilon_k/\alpha_k, \varepsilon_k'/\alpha_k')
$
for both $k=1,2$,
where $\varepsilon_k, \varepsilon_k'\in\{\pm 1\}$,
then $S_i$ and $S_j$ are not isotopic
by Lemmas \ref{distinction-hs-1}, \ref{distinction-hs-2}
and Theorem \ref{prop-hs-3b-0}.
\qed
%\end{proof}

\begin{remark}
{\rm Lemma \ref{distinction-hs-1} enables us to complete Table 5.2 in \cite{Mor}.
Except for the following cases, 
each number $\mu$ on the table 
gives the exact number of Heegaard splittings up to isotopy.}
\end{remark}

\begin{itemize}
\item 
Let $M_1=E(S(2n+1,1))=D(1/2, -n/(2n+1))$ and $M_2=D(1/2, -1/3)$, 
and suppose that a regular fiber of $M_2$ is identified with a loop $m_1h_1^a$, 
where $m_1$ and $h_1$ are, respectively, a meridian and a regular fiber of $M_1$.
Then $M=M_1\cup_f M_2$ admits exactly two genus-2 Heegaard splittings up to isotopy, 
one of which belongs to F(1) and the other belongs to F(2-2).

\item 
Let $M_1=D(1/2, -1/3)$ and $M_2=E(S(2n+1,1))=D(1/2, -n/(2n+1))$, 
and suppose that a regular fiber of $M_1$ is identified with a loop $m_2h_2^a$, 
where $m_2$ and $h_2$ are, respectively, a meridian and a regular fiber of $M_2$.
Then $M=M_1\cup_f M_2$ admits exactly two genus-2 Heegaard splittings up to isotopy, 
one of which belongs to F(1) and the other belongs to F(2-1).

\item 
When one of $M_1$ and $M_2$ is $D(1/2, -n/(2n+1))$ and the other is homeomorphic to $KI$,
any genus-2 Heegaard splitting which belongs to F(3)
is isotopic to a genus-2 Heegaard splitting in F(1).

\end{itemize}

Moreover, one can obtain the homeomorphism classification of 3-bridge presentations 
and genus-2 Heegaard splittings 
by considering the action of the mapping class group of $M$ on the Heegaard surfaces.

%%%%%%%%%%%%%%%%%%%%%%%%%%%%%%%%%%%%%%%%%%%%%%%%%%%%%%

\section{3-bridge spheres for Montesinos links}\label{sec-mont}

In this section, we prove Theorem \ref{thm-mont}.

Let $M=S^2(e_0; \beta_1/\alpha_1, \beta_2/\alpha_2, \beta_3/\alpha_3)$
be a Seifert fibered space over $S^2$ with three exceptional fibers.
To describe the results of \cite{Boi},
we take two exceptional fibers 
$\eta_i$, $\eta_j$ $(1\leq i\neq j\leq 3)$ 
and connect them by an arc projected to a simple arc on the base $S^2$.
A regular neighborhood $V(i,j)$ of the graph obtained 
is a handlebody of genus 2.
The closure $W(i,j)$ of the complement is 
also a handlebody of genus 2 
and we obtain a Heegaard surface 
$F(i,j)=\partial V(i,j)=\partial W(i,j)$ of $M$.
This is called a {\it vertical Heegaard surface}.

\begin{theorem}\label{small-seifert}
(\cite[Theorem 2.5]{Boi})
Let $M=S^2(e_0; \beta_1/\alpha_1, \beta_2/\alpha_2, \beta_3/\alpha_3)$
be a Seifert fibered space over $S^2$ with three exceptional fibers.
\begin{itemize}
\item[{\rm (A)}] 
If $\beta_i\not\equiv \pm 1 \pmod{\alpha_i}$ for $i=1,2,3$, 
then $M$ admits, up to isotopy, exactly three Heegaard surfaces of genus 2,
namely $F(1,2)$, $F(2,3)$, $F(3,1)$.
\item[{\rm (B)}] 
If $\beta_i\not\equiv \pm 1 \pmod{\alpha_i}$ for $i=1,2$
and $\beta_3\equiv \pm 1 \pmod{\alpha_3}$, 
then $M$ admits, up to isotopy, exactly two Heegaard surfaces of genus 2,
namely $F(1,2)$ and $F(2,3)(=F(3,1))$.
\item[{\rm (C)}] 
If $\beta_i\equiv \pm 1 \pmod{\alpha_i}$ for $i=2,3$, 
then $M$ admits, up to isotopy, a single Heegaard surface of genus 2,
$F(1,2)(=F(2,3)=F(3,1))$, 
except when $M$ is one of 
$S(-\frac{1}{6a}; \frac{1}{2}, \frac{(-a)^{-1}}{3},$ $\frac{6^{-1}}{a})$ ($a$ is odd), 
$S(-\frac{1}{6a}; \frac{(-1)^{-1}}{3}, \frac{(-a)^{-1}}{3}, \frac{3^{-1}}{a})$ and
$S(-\frac{1}{4b}; \frac{1}{2}, \frac{(-b)^{-1}}{4}, \frac{4^{-1}}{a})$,
where $a\geq 7$, $b\geq 5$ and $g.c.d.(a,3)=g.c.d.(b,2)=1$.
In each exceptional case
$M$ admits, up to isotopy, 
a unique additional Heegaard surface of genus 2 
obtained by presenting $M$ as the double branched covering of $S^3$ 
branched along a 3-bridge presentation of a link 
in Figure \ref{fig-mont-links-exception}.
(see \cite{Bir2}, \cite{Boi4}).
\end{itemize}
%
%%%%%%%%%
\begin{figure}
\begin{center}
\includegraphics*[width=9.5cm]{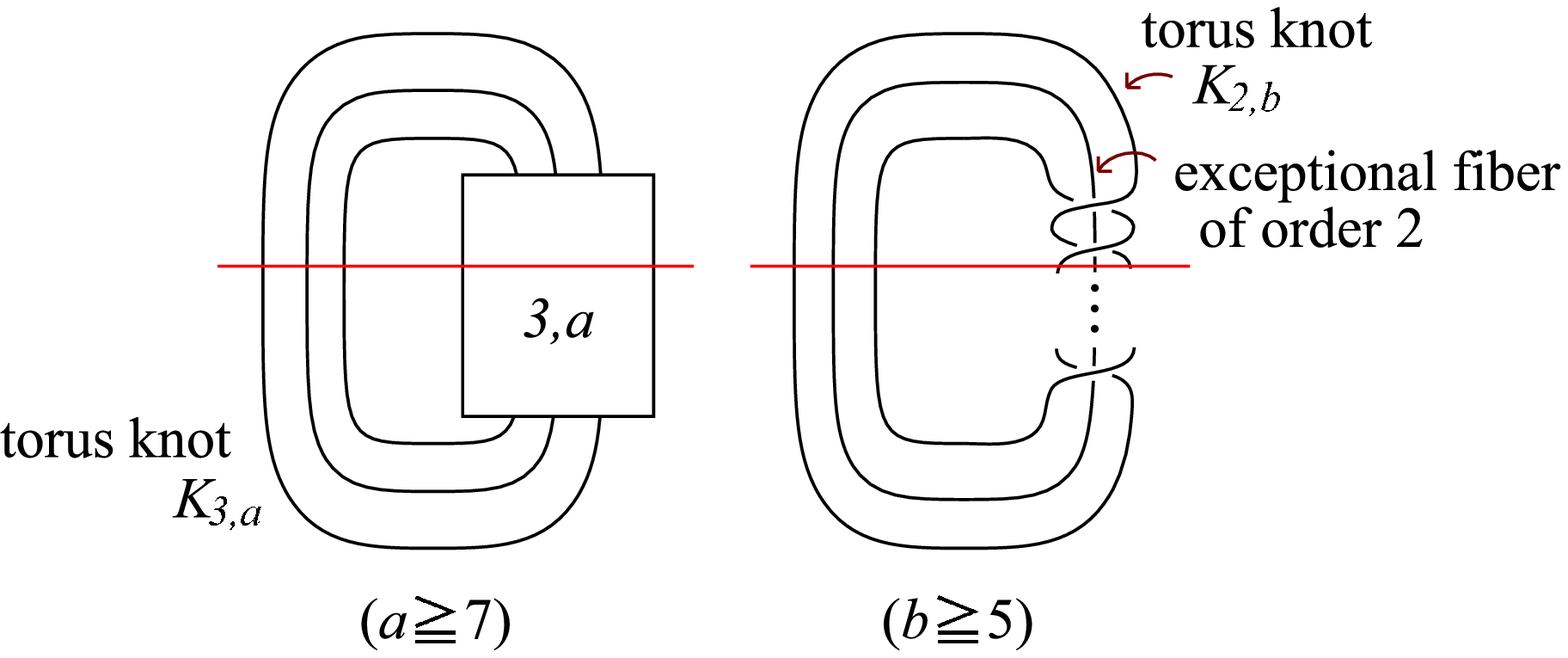}
\end{center}
\caption{}
\label{fig-mont-links-exception}
\end{figure}
%%%%%%%%%
%
\end{theorem}

\begin{lemma}\label{lemma-exception}
The links in Figure \ref{fig-mont-links-exception}
are not arborescent links.
\end{lemma}

\begin{proof}
Suppose that a link, say $L$, in Figure \ref{fig-mont-links-exception}
is an arborescent link.
Since $L$ is not hyperbolic,
it must be equivalent to a link in Figure \ref{fig-alink2}
%
%%%%%%%%%%%%%
\begin{figure}
\begin{center}
\includegraphics*[width=9.5cm]{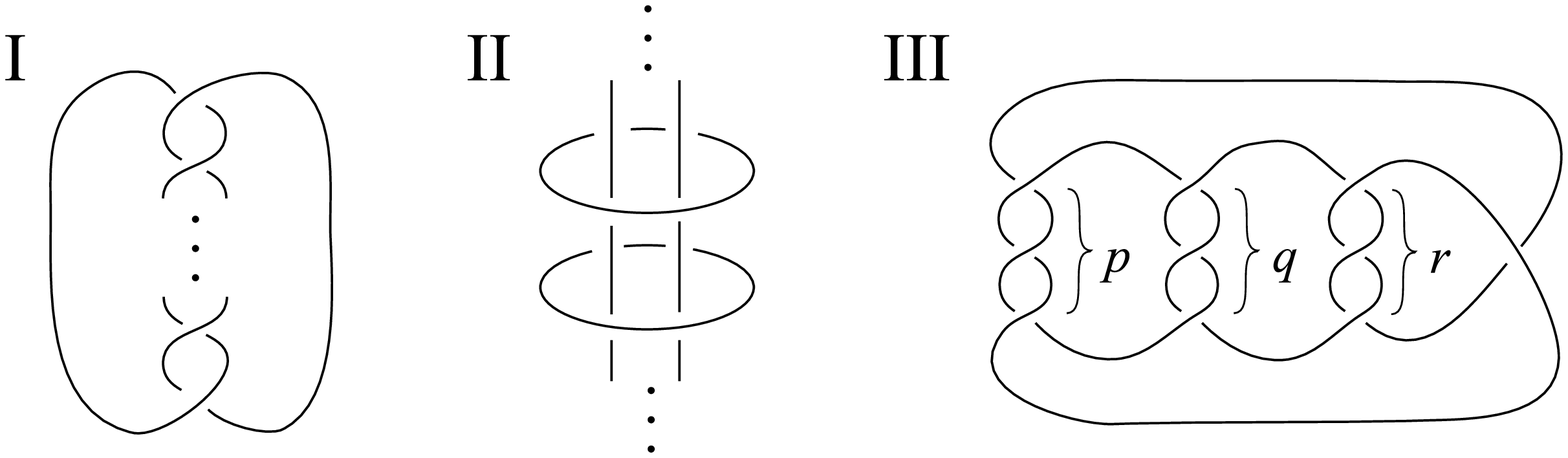}
\end{center}
\caption{}
\label{fig-alink2}
\end{figure}
%%%%%%%%%%%%%
%
by \cite{Bon, Fut} (cf. \cite[Proposition 3]{Jan2}).
Namely, one of the following holds.
\begin{itemize}
\item[I.] $L$ is the boundary of a single unknotted band,
i.e., a torus knot or link of type $(2,n)$ for some $n\in\Z$. 
\item[II.]  $L$ has two parallel components, 
each of which bounds a twice-punctured disk 
properly embedded in $S^3\setminus L$.
\item[III.]  $L$ or its reflection is the pretzel link $P(p,q,r,-1)$, 
where $p,q,r\geq 2$ and 
$\displaystyle\frac{1}{p}+\displaystyle\frac{1}{q}+\displaystyle\frac{1}{r}\geq 1$.
\end{itemize}

However, 
this cannot occur
since
\begin{itemize}
\item
a link in I of Figure \ref{fig-alink2} is either a torus knot $K_{2,n}$ ($n$ is odd) or
a torus link $K_{2,n}$ ($n$ is even) which consists of two trivial components,\\[-5mm]
\item
a link in II of Figure \ref{fig-alink2} consists of at least three components,\\[-5mm]
\item
$P(2,2,n,-1)$ ($n$ is odd) is a union of a torus knot $K_{2,n}$
and its core of index $n$,\\[-5mm]
\item
$P(2,2,n,-1)$ ($n$ is even) has three components,\\[-5mm]
\item
$P(2,3,3,-1)$ is the torus knot $K_{3,4}$,\\[-5mm]
\item
$P(2,3,4,-1)$ is a union of the torus knot $K_{2,3}$
and its core of index 2,\\[-5mm]
\item
$P(2,3,5,-1)$ is the torus knot $K_{3,5}$,\\[-5mm]
\item
$P(2,3,6,-1)$ consists of two components, the torus knot $K_{2,3}$ and the unknot,\\[-5mm]
\item
$P(2,4,4,-1)$ has three components,\\[-5mm]
\item
$P(3,3,3,-1)$ consists of two trivial components.
\end{itemize}
Hence we obtain the desired result.
\end{proof}
\vspace{2mm}

\begin{remark}\label{rmk-exception}
{\rm
Let $F$ be an exceptional Heegaard surface of a manifold $M$ in Theorem \ref{small-seifert} (C)
and $\tau_F$ the hyper-elliptic involution $\tau$ associated with $F$.
Then $(M,\fix(\tau_F))/\tau_F$ is a links in Figure \ref{fig-mont-links-exception}.
Hence, for any arborescent link $L$,
the covering involution $\tau_L$ of $M_2(L)$ is not equivalent to $\tau_F$.
}
\end{remark}

To prove Theorem \ref{thm-mont} (2), we need the following proposition.

\begin{proposition}\label{sakuma-elliptic}
(\cite[Theorem 4.1]{Sak})
Let $L$ be an elliptic Montesinos link and 
assume that $L$ is not a 2-bridge link. 
Then the symmetry group $\sym (S^3,L)$ %$=\pi_0(\isom O(L))$
is as follows according to the type of $L$.
Here, $\psi_i$ $(i=1,2,3,4)$ is a symmetry of $(S^3,L)$
as illustrated in Figure \ref{fig-symmetries}.
%
%%%%%%%%%
\begin{figure}
\begin{center}
\includegraphics*[width=12cm]{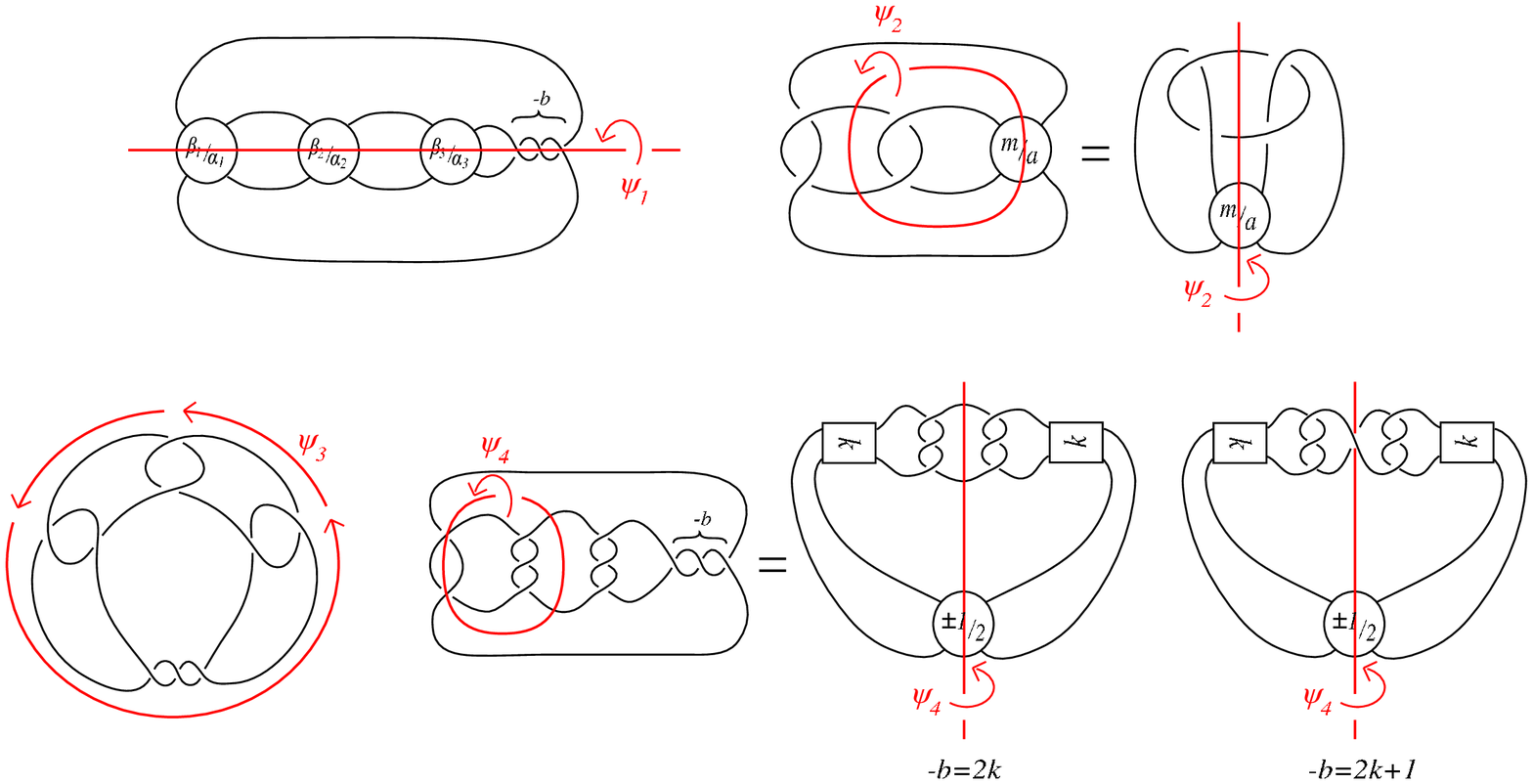}
\end{center}
\caption{}
\label{fig-symmetries}
\end{figure}
%%%%%%%%%
%
\begin{itemize}
\item[{\rm (i)}]
Let $L=L(b; 1/2, 1/2, \beta/\alpha)$ and put $m=(-b+1)\alpha+\beta$.

{\rm (i-1)} If $g.c.d.(m,2\alpha)=1$, 
then $\sym (S^3,L)$ is given by
%
%\begin{table}
\renewcommand{\arraystretch}{1.1} %sŠÔ'²ß
\begin{center}
\begin{tabular}{|c||c|c|}
\hline
{} & {$\alpha\geq 3$} & {$\alpha=2$} 
\\ \hline\hline
{$m\neq 1$} & {$\langle\psi_1, \psi_2\rangle\cong\Z_2\oplus\Z_2$} & 
{$\langle\psi_1, \psi_3\rangle\cong\Z_2\oplus D_3$} 
\\ \hline
{$m=1$} & 
{$\langle\psi_1(=\psi_2)\rangle\cong\Z_2$ if $\alpha$ is odd,} 
& {$\langle\psi_1, \psi_3\rangle\cong\Z_2\oplus D_3$} 
\\ %\hline
{} & {$\langle\psi_1, \psi_2\rangle\cong\Z_2\oplus\Z_2$ if $\alpha$ is even.} & {}
\\ \hline
\end{tabular}
%\caption{}
%\label{table1}
\end{center}
%\end{table}
% 

{\rm (i-2)} If $m$ is even and $g.c.d.(m,\alpha)=1$, 
then $$\sym (S^3,L)=\langle\psi_1, \psi_2\rangle\cong \Z_2\oplus\Z_2.$$

\item[{\rm (ii)}]
Let $L=L(b; 1/2, \beta_2/3, \beta_3/3)$ and
put $m=-6b+3+2(\beta_2+\beta_3)$.
Then 
\begin{eqnarray*}
\sym (S^3,L)=\left\{
\begin{array}{ll}
\langle\psi_1, \psi_4\rangle\cong\Z_2\oplus\Z_2 & if\ g.c.d.(m,12)=1\ and\ m\neq 1,\\
\langle\psi_1\rangle\cong\Z_2 & otherwise.
\end{array}
\right.
\end{eqnarray*}
\item[{\rm (iii)}]
If $L=L(b; 1/2, \beta_2/3, \beta_3/4)$
or $L(b; 1/2, \beta_2/3, \beta_3/5)$, then
$$\sym (S^3,L)=\langle\psi_1\rangle\cong \Z_2.$$
\end{itemize}
\end{proposition}

{\sc Proof of Theorem \ref{thm-mont}.}
%\begin{proof}{\bf of Theorem \ref{thm-mont}}
Let $L$ be a 3-bridge Montesinos link.
Let $M_2(L)$ be the double branched covering of $S^3$ branched over $L$, and
let $p:M_2(L)\rightarrow S^3$ be the covering projection.

(1) Suppose that $L$ is nonelliptic, and  
let $P_i$ ($i=1,2,3,4,5,6$) be a 3-bridge sphere of $L$ in Figure \ref{fig-mont-links}.
Then the pre-images $p^{-1}(P_1)$, $p^{-1}(P_3)$ and $p^{-1}(P_5)$ 
are isotopic to $F(1,2)$, $F(2,3)$ and $F(1,3)$, respectively.
By Theorem \ref{small-seifert} (\cite[Theorem 2.5]{Boi}) and Remark \ref{rmk-exception},
these are the only genus-2 Heegaard surfaces of $M_2(L)$ 
whose hyper-elliptic involutions are strongly equivalent to $\tau_L$.
On the other hand, we see $P_{i+1}=\rho(P_{i})$ for each $i=1,3,5$,
where $\rho$ is the symmetry of $(S^3,L)$ given in Figure \ref{fig-mont-inv}.
Hence, by Theorem \ref{prop-hs-3b-0} and Remark \ref{remark-tau}, 
we see that $L$ admits at most six 3-bridge spheres $P_1$,\dots,$P_6$ up to isotopy.

(2) Suppose that $L$ is elliptic. 
Let $P_1$ be the 3-bridge sphere of $L$ 
as illustrated in Figure \ref{fig-mont-links}, and 
let $P$ be any 3-bridge sphere of $L$.
Set $F_1:=p^{-1}(P_1)$ and $F:=p^{-1}(P)$.
We note that $\tau_{F_1}=\tau_F=\tau_L$.
Since $M_2(L)$ admits a unique genus-2 Heegaard surface up to isotopy
whose hyper-elliptic involution is $\tau_L$, 
by Theorem \ref{small-seifert} (\cite[Theorem 2.5]{Boi}) and Remark \ref{rmk-exception},
$F$ is isotopic to $F_1$.
Thus, there exists a self-homeomorphism $\varphi$ of $M_2(L)$ 
such that $\varphi(F_1)=F$ and $\varphi$ is isotopic to the identity.
By the proof of \cite[Theorem 8]{Bir} 
(cf. the proof of \cite[Proposition 5]{Jan2}),
we may assume that $\varphi$ is $\tau_L$-equivariant, 
where $\tau_L$ is the covering transformation.
So we have a self-homeomorphism $\psi$ of $(S^3, L)$ sending $P_1$ to $P$.
Hence, it suffices to show that generators of the symmetry group $\sym(S^3,L)$
preserve $P_1$ up to isotopy.

We show this only when $L$ satisfies the condition (i-1) of Proposition \ref{sakuma-elliptic}, where $m=1$ and $\alpha\geq 3$.
(The other cases can be treated similarly.)
In this case, we have $\frac{\beta}{\alpha}=\frac{1}{\alpha}+(b-1)$
from $(-b+1)\alpha+\beta=1$, and hence,
$L=L(b;1/2, 1/2, \beta/\alpha)$ is equivalent to $L(0;-1/2, 1/2, 1/\alpha)$.
%
%%%%%%%%%
\begin{figure}
\begin{center}
\includegraphics*[width=11.5cm]{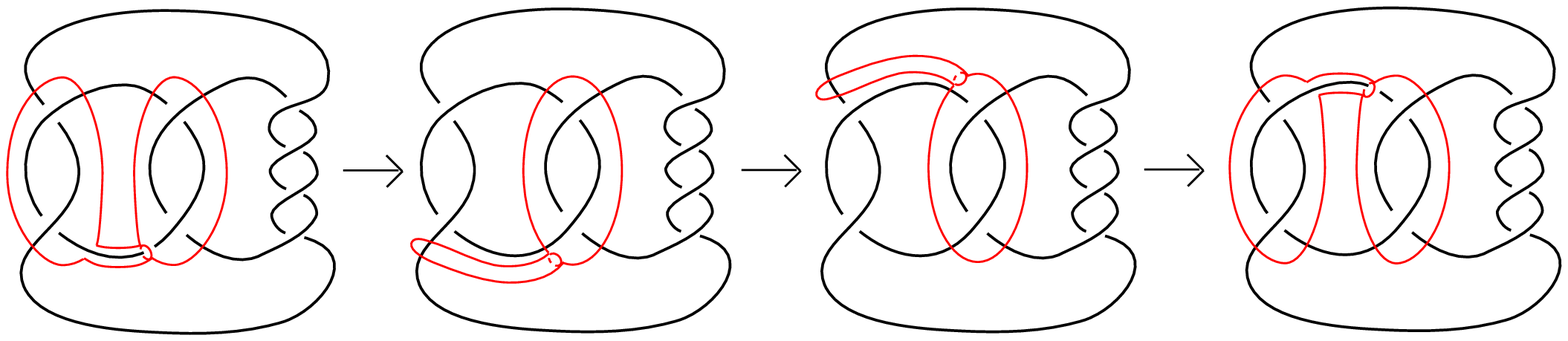}
\end{center}
\caption{}
\label{fig-isotopy-2}
\end{figure}
%%%%%%%%%
%
Note that $\psi_2(P_1)=P_1$
and that we can isotope $\psi_1(P_1)$ to $P_1$  
as illustrated in Figure \ref{fig-isotopy-2}.
Hence, $L$ admits a unique 3-bridge sphere up to isotopy.

An isotopy between $P$ and $P_1$ can be constructed similarly
for every case.
Thus every elliptic Montesinos link admits a unique 3-bridge sphere up to isotopy.
\qed
%\end{proof}

\begin{remark}\label{rmk-nonelliptic}
{\rm 
For nonelliptic Montesinos links, 
we give some conditions for $P_i$ and $P_j$ ($i,j=1,\dots,6$, $i\neq j$) 
to be isotopic
by using isotopies as in the proof of Theorem \ref{thm-mont}.
The following table gives the conditions for each pair,
where (1-$k$) and (2-$k$) $(k=1,2,3)$
denote the following conditions.

{\rm (1-$k$)} $\beta_k\equiv \pm 1 \pmod{\alpha_k}$ $(k=1,2,3)$,

{\rm (2-$k$)} $\alpha_k=2$ $(k=1,2,3)$.

$$
{\vbox
{\tabskip=0pt\offinterlineskip 
 \halign{\strut#& 
  \vrule# & \hfil#\hfil & 
  \vrule# & \hfil#\hfil & 
  \vrule# & \hfil#\hfil & 
  \vrule# & \hfil#\hfil & 
  \vrule# & \hfil#\hfil & 
  \vrule# & \hfil#\hfil & 
  \vrule#
 \cr
 \noalign{\hrule}
  & & { }$P_2${ } & & { }$P_3${ } 
  & & { }$P_4${ } & & { }$P_5${ } & & { }$P_6${ } & & &
 \cr 
 \noalign{\hrule}
  & & { }(2-1) or (2-2) & &  & & { }(1-2) & &  & & { }(1-1) & &  { }$P_1$ & 
 \cr
 \noalign{\hrule}\multispan3 
  & &\strut { }(1-2) & &  & & { }(1-1) & &  & &  { }$P_2$ & 
 \cr \multispan3 
  & & \hrulefill & &\hrulefill & & \hrulefill & & \hrulefill & & \hrulefill & 
 \cr \multispan5 
  & & \strut  { }(2-2) or (2-3) & &  & & { }(1-3) & &  { }$P_3$ & 
 \cr \multispan5 
  & & \hrulefill & & \hrulefill & & \hrulefill & & \hrulefill & 
 \cr \multispan7 
  & & \strut { }(1-3) & &  & &  { }$P_4$ & 
 \cr \multispan7 
  & & \hrulefill & & \hrulefill & &  \hrulefill & 
 \cr \multispan9 
  & & \strut  { }(2-1) or (2-3) & &  { }$P_5$ & 
 \cr \multispan9 
  & & \hrulefill & &  \hrulefill & 
% \cr \multispan{13} 
%  \strut 
% \cr \multispan{13} 
%  \hfil Table A\hfil 
 \cr 
}}}
$$

For example, $P_1$ and $P_2$ are isotopic 
if (2-1) $\alpha_1=2$ or (2-2) $\alpha_2=2$ holds.
Moreover, this implies, for example, $P_1$ and $P_3$ are isotopic
if (i) (2-1) (or (2-2)) and (1-2) holds, 
(ii) (1-2) and (2-2) (or (2-3)) holds, or
(iii) (1-1) and (1-3) holds.
If $\beta_i\equiv \pm 1 \pmod{\alpha_k}$ for all $i=1,2,3$ 
and $b=\Sigma \frac{\beta_i}{\alpha_i}-\Sigma \frac{\pm 1}{\alpha_i}$,
then $P_1,\dots, P_6$ are mutually isotopic.
}
\end{remark}

%%%%%%%%%%%%%
%\ \\[-1mm]
{\bf Acknowledgment.}
The author would like to express her appreciation to Professor
Makoto Sakuma for his guidance, advices and encouragement.
She would also like to thank 
Kanji Morimoto and Kai Ishihara 
for their helpful comments.

%%%%%(end)%%%%%%%%%%%%%%%%%%%%%%%%%%%%%%%%%%%%%%%%%%%%%%%%%%%%%%%%%%%%%%%%%%%%%%%

%% \begin{thebibliography}{9, 99 or Abc99}
%\begin{thebibliography}{9}  for 1-digit labels
%% \begin{thebibliography}{99}  for 2-digit labels
%% \begin{thebibliography}{Abc}  for alphanumeric labels
%\begin{thebibliography}{Abc}

\vspace{1cm}

%\profile{Yeonhee Jang}
%{Department of Mathematics\\ Hiroshima University \\
%Higashi-Hiroshima\\
%Hiroshima 739-8526, Japan\\
%yeonheejang@hiroshima-u.ac.jp}

\label{finishpage}

\end{document}